\documentclass[11pt]{amsart}
\usepackage[T1]{fontenc}
\usepackage{geometry}
\usepackage[english]{babel}
\usepackage{amsmath,amsfonts,amssymb,amsthm,amscd}
\usepackage{dsfont}
\usepackage{a4wide}
\usepackage[all]{xy}
\usepackage{graphicx}
\usepackage{float}
\usepackage{subfig}		
\usepackage{color}
\usepackage[normalem]{ulem}									
\numberwithin{equation}{section}
\def\Q{\mathbb Q}
\def\Z{\mathbb Z}
\def\C{\mathbb C}

\newcommand{\Zz}{\mathbb{Z}}
\newcommand{\Cc}{\mathbb{C}}

\newcommand{\Ss}{\mathbb{S}}

\renewcommand{\epsilon}{\varepsilon}
\newcommand{\sign}{\operatorname{sign}}
\newcommand{\inte}{\operatorname{int}}
\newcommand{\lk}{\operatorname{\ell\mathit{k}}}

\newtheorem{thm}{Theorem}[section]

\newtheorem{propo}[thm]{Proposition}
\newtheorem{lemma}[thm]{Lemma}
\newtheorem{add}[thm]{Addendum}

\newtheorem{cor}[thm]{Corollary}

\theoremstyle{definition}
\newtheorem{rmk}[thm]{Remark}
\newtheorem{de}[thm]{Definition}
\newtheorem{ex}[thm]{Example}

\let\TN=T
\def\iTN{\mathring{\TN}}
\def\CT{\mathcal{T}}
\def\torus{\mathring{\mathcal{T}}}
\def\Log{\operatorname{Log}}
\def\ind{\operatorname{ind}}
\def\const{\operatorname{const}}
\def\sg{\operatorname{sg}}
\def\Ker{\operatorname{Ker}}
\def\trace{\operatorname{trace}}
\def\nl{\operatorname{null}}
\let\<\langle
\let\>\rangle
\let\onto\twoheadrightarrow
\let\into\hookrightarrow
\def\RR{\mathbb{R}}
\def\ZZ{\mathbb{Z}}
\def\HH{\mathcal{H}}
\def\Zm{\ZZ/m}
\def\Zn{\ZZ/n}
\def\1{^{-1}}
\let\ARC\bar
\let\ARC\mathbf
\def\kk{\ARC{k}}
\def\ll{\ARC{l}}
\def\ee{\ARC{e}}
\def\ff{\ARC{f}}
\def\cc{\ARC{c}}
\def\aa{\ARC{a}}

\def\tomega{\tilde\omega}
\def\vrho{{\vect\rho}}
\let\defect\delta

\def\tlambda{\lambda^{\fam0\!w}}
\let\storus\tau
\def\btorus{\tilde\storus}

\def\KK{\mathcal K}
\def\DD{\mathcal D}

\def\vlk{\operatorname{\,\overline{\!\lk}}}

\def\iif{\<\,{\cdot}\,,\,{\cdot}\,\>}
\def\iref#1{(\ref{#1})}

\def\vect{}

\def\Dg:{\endgraf{\bf Dg:}\enspace\ignorespaces}
\def\3{\color{red}}
\def\4{\color{magenta}}
\def\5{\color{cyan}}
\def\7{\color{blue}}
\def\9{\color{green}}

\title{The signature of a splice}

\author{Alex Degtyarev}

\address{%
Department of Mathematics\\
Bilkent University\\
06800 Ankara, Turkey}

\email{degt@fen.bilkent.edu.tr}

\author{Vincent Florens}

\address{%
Laboratoire de Math\'{e}matiques et leurs applications, UMR CNRS 5142\\
Universit\'{e} de Pau et des Pays de l'Adour\\
Avenue de l'Universit\'{e}\\
 BP 1155 64013 Pau Cedex, France
}

\email{vincent.florens@univ-pau.fr}

\author{Ana G.\ Lecuona}

\address{%
Aix Marseille Universit\'{e}, CNRS, Centrale Marseille, I2M, UMR 7373, 13453 Marseille, France}

\email{ana.lecuona@univ-amu.fr}

\begin{document}

\begin{abstract}
We study the behavior of the signature of  colored links \cite{F,CF}   under the splice operation. We extend the construction to colored links in
integral homology spheres and show that the signature is almost additive,
with a correction term independent of the links.
We interpret this correction term as the signature of a
generalized Hopf link and give a simple closed formula to compute it.
\end{abstract}

\maketitle

\let\thefootnote\relax \footnote{The first author was partially supported by the JSPS grant L-15517 and T\"{U}B\.ITAK grant 114F325, the second  by the ANR Project Interlow
 JCJC-0097-01 and the third by Spanish GEOR MTM2011-22435. }

\section{Introduction}

The \emph{splice} of two links is an operation defined by Eisenbud and Neumann
in~\cite{EN}, which generalizes several other operations on links such as connected sum, cabling, and
disjoint union. The precise definition is given in Section~\ref{s:set-up}
(see Definition~\ref{def:splice}),
but the rough idea is as follows: the splice
of two links $K'\cup L'\subset\Ss'$ and $K''\cup L''\subset\Ss''$
along the
distinguished
components $K'$ and $K''$ is the link
$L'\cup L''$ in the $3$-manifold~$\Ss$
obtained by an appropriate gluing of the exteriors
of~$K'$ and~$K''$.
There has been much interest in understanding
the behavior of various
link invariants under the splice
operation. For example, the genus and the fiberability of a link are
additive, in a suitable sense, under splicing \cite{EN}. The behavior of the
Conway polynomial has been studied in \cite{C}, and more recently the
relation between the $L$-spaces in Heegaard--Floer homology and splicing has been
addressed in \cite{HL}.
The goal of this paper is   to obtain a similar (non-)additivity statement for
the multivariate signature of oriented colored links. As a consequence, we
show that the conventional univariate
Levine--Tristram signature of a splice depends on
the \emph{multivariate} signatures of the summands.

In Section~\ref{s:defs} we define the \emph{signature} of a colored link
in an integral homology sphere.
This is a natural generalisation of the multivariate
 extension of the Levine-Tristram signature
 of a link in the $3$-sphere, considered in
 \cite{F,CF}.
The principal result of the paper is Theorem~\ref{t:main'},
expressing the
signature of the splice of two links in terms of the signatures of the
summands.
We show that the signature is almost additive: there is a defect, but it
depends only on some combinatorial data of the links (linking numbers),
and not on the links themselves.
Geometrically, this defect term appears as the multivariate
signature of a certain generalized Hopf link, which is computed in
Theorem~\ref{t:hopf}.
At the end of Section~\ref{s:results},
we discuss a few applications of
Theorem~\ref{t:main'} and relate it to some previously known results: namely,
we compute the signature of a satellite knot (see Section~\ref{s:satellite}
and
Theorem~\ref{cor:satellite}) and that of an iterated torus link (see
Section~\ref{s:torus} and Theorem~\ref{cor:torus}).
More precisely, we reduce the computation to the signature of cables over the unknot.
 We also show that the multivariate signature of a link can be computed by means of
the conventional Levine--Tristram signature of an auxiliary link
(see Section~\ref{s.univariate} and Theorem~\ref{th.many-one}).

The paper is organized as follows. Section~\ref{s:results} is devoted to the detailed statement of main
 results, and the computation of the defect. In Section~\ref{s:prel}, we introduce the
 necessary  background material on twisted intersection forms and construct
 the signature of colored links in integral homology spheres.
The proofs of the main theorems are carried out in Section~\ref{s:proofs} and Section~\ref{s:Hopf}, where the signature of the generalized Hopf links is computed.

\subsection*{Acknowledgements}
We would like to thank S.~Orevkov, who brought the problem to our
attention. We are also grateful to the anonymous referees of this paper who
corrected a mistake in the original version of Corollary~\ref{tristram}
and a sign in Theorems~\ref{t:main'} and~\ref{t:hopf};
Example~\ref{ex.referee} was also suggested by a referee.
This work was partially
completed during the first and third authors' visits to
the University of Pau, supported by
 the CNRS,
 and the first author's visit to the Abdus Salam International
Centre for Theoretical Physics.

\section{Principal results}\label{s:results}

\subsection{The set-up}\label{s:set-up}
 A \emph{$\mu$-colored link}
is an oriented link~$L$ in an integral homology sphere~$\Ss$ equipped with a
surjective function $\pi_0(L)\onto\{1,\ldots,\mu\}$, referred to as the
\emph{coloring}.
The union of the components of~$L$ given the same color $i=1,\ldots,\mu$ is
denoted by $L_i$.

The \emph{signature} of a $\mu$-colored link~$L$ is a
certain $\Zz$-valued function $\sigma_L$ defined on the \emph{character
torus}
\begin{equation}
\CT^\mu:=\bigl\{(\omega_{1},\dots,\omega_{\mu})\in(S^1)^\mu\subset\Cc^\mu\bigm|
 \omega_{j}=\exp(2\pi i\theta_{j}),\ \theta_{j}\in\Q\bigr\},
\label{eq:torus}
\end{equation}
see Definition~\ref{sign} below for details.
We let $\CT^0:=\{1\}\in\Cc$.
Note that $\CT^\mu$ is an abelian group. If $\mu=1$,
 the link $L$ is monochrome and $\sigma_L$
 coincides with   the restriction  (to rational points)  of the Levine--Tristram signature \cite{b:Tri}
(whose definition in terms of Seifert form
extends naturally
to links in homology spheres).
 Given a character $\omega\in\CT^\mu$ and a vector $\lambda\in\Z^\mu$, we
use the common notation
$\omega^\lambda:=\prod_{i=1}^\mu\omega_i^{\lambda_i}$.

Often, the components of~$L$ are split naturally into two groups,
$L=L'\cup L''$, on which the coloring takes, respectively, $\mu'$ and $\mu''$
values, $\mu'+\mu''=\mu$. In this case, we regard $\sigma_L$ as a function of
two ``vector'' arguments $(\omega',\omega'')\in\CT^{\mu'}\times\CT^{\mu''}$.
We use this notation freely, hoping that each time its precise meaning is
clear from the context.

Clearly, in the definition of colored link, the precise set of colors is not
very important; sometimes, we also admit the color~$0$. As a special case, we
define a \emph{$(1,\mu)$-colored link}
$$
K\cup L=K\cup L_1\cup\ldots\cup L_\mu
$$
as a $(1+\mu)$-colored link
in which $K$ is the only component given the distinguished
color~$0$.
Here,  we assume $K$ connected;
this component, considered distinguished, plays a special role in a number
of operations.

In the following definition, for a $(1,\mu^*)$-colored link
$K^*\cup L^*\subset\Ss^*$, $*=\prime$ or $\prime\prime$, we denote by
$T^*\subset\Ss^*$ a
small tubular neighborhood of~$K^*$ disjoint from~$L^*$ and let
$m^*,\ell^*\subset\partial T^*$ be, respectively, its meridian and longitude.
(The latter is well defined as $\Ss^*$ is a homology sphere.)

\begin{de}\label{def:splice}
Given two $(1,\mu^*)$-colored links $K^*\cup L^*\subset\Ss^*$,
$*=\prime$ or $\prime\prime$, their \emph{splice}
is the $(\mu'+\mu'')$-colored link $L'\cup L''$ in the integral
homology sphere
$$
\Ss:=(\Ss'\smallsetminus\inte\TN')\cup_\varphi(\Ss''\smallsetminus\inte\TN''),
$$
where the gluing homeomorphism $\varphi\colon\partial\TN'\to\partial\TN''$
takes~$m'$ and $\ell'$ to~$\ell''$ and~$m''$, respectively.
\end{de}

\subsection{The signature formula}\label{s:notation}

Given a list (vector, etc.) $a_1,\ldots,a_i,\ldots,a_n$, the notation
$a_1,\ldots,\hat a_i,\ldots,a_n$ designates that the $i$-th
element (component, etc.) has been removed.
The complex conjugation is denoted by
$\eta\mapsto\bar\eta$.
The same notation applies to
the elements of the character torus~$\CT^\mu$, where we have
$\bar\omega=\omega\1$.

The \emph{linking number} of two disjoint oriented circles~$K$, $L$ in an
integral homology sphere~$\Ss$ is denoted by
$\lk_{\Ss}(K,L)$,
with $\Ss$ omitted whenever understood.
For a $(1,\mu)$-colored link $K\cup L$, we also define the \emph{linking
vector} $\vlk(K,L)=(\lambda_1,\ldots,\lambda_\mu)\in\Zz^\mu$, where
$\lambda_i:=\lk(K,L_i)$.

The \emph{index} of a real number $x$ is
 defined via $\ind(x):=\lfloor x\rfloor-\lfloor-x\rfloor\in\ZZ$.
The \emph{$\Log$-function}
$\Log\colon\CT^1\to[0,1)$ sends $\exp(2\pi it)$
to $t\in[0,1)$.
This
function extends to $\Log\colon\CT^\mu\to[0,\mu)$ via
$\Log\vect\omega=\sum_{i=1}^\mu\Log\omega_i$;
in other
words, we specialize each argument to the interval $[0,1)$ and add the
arguments \emph{as real numbers}
(rather than elements of $\CT^1$)
afterwards.
For any integral vector $\vect\lambda\in\Zz^\mu$, $\mu\ge0$, we define the \emph{defect function}
\begin{align*}
\textstyle
\defect_\lambda \colon\CT^\mu&\longrightarrow\Z \\
 \vect\omega&\longmapsto\textstyle
 \ind\bigl(\sum_{i=1}^\mu \lambda_i \Log\omega_i\bigr)-\sum_{i=1}^\mu \lambda_i \ind(\Log\omega_i).
\end{align*}
For short, if $\lambda_i=1$ for all $i$, we simply denote the defect $\defect$, and omit the subscript. The reader is referred to Figure~\ref{fig:delta} for a few examples of the defect function on $\CT^{2}$.

\begin{figure}\centering
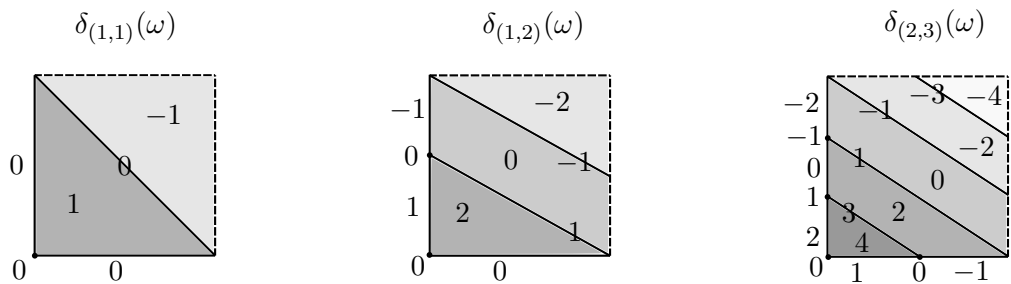
\caption{The values of three defect functions for $\omega\in\CT^{2}$. The defect is constant on the shaded regions and on the interior of the segments dividing the squares. The values of the defect in the extremal cases, $\omega_{1}=1$ or $\omega_{2}=1$, are given by the numbers on the left and bottom of the squares respectively.}
\label{fig:delta}
\end{figure}

The following statement is the principal result of the paper.

\begin{thm}\label{t:main'}
For $*=\prime$ or $\prime\prime$, consider a $(1,\mu^*)$-colored link
$K^*\cup L^*\subset\Ss^*$, and let $L\subset\Ss$ be the splice of the two
links.
For characters $\omega^*\in\CT^{\mu^*}$,
introduce the notation
$$
\lambda^*:=\vlk(K^*,L^*)\in\Zz^{\mu^*},\qquad
 \upsilon^*:=(\omega^*)^{\lambda^*}\in\CT^1.
$$
Then, assuming that $(\upsilon',\upsilon'')\ne(1,1)$, one
has
$$
\sigma_L(\vect\omega',\vect\omega'')=
 \sigma_{K'\cup L'}(\upsilon'',\vect\omega')+
 \sigma_{K''\cup L''}(\upsilon',\vect\omega'')
 +\defect_{\lambda'}(\vect\omega')\defect_{\lambda''}(\vect\omega'').
$$
\end{thm}

\begin{rmk}
Eisenbud and Neumann \cite[Theorem 5.2]{EN} showed that the Alexander polynomial is multiplicative under the splice. For a $\mu$-colored link $L$, we denote $\Delta_L(t_1,\dots,t_\mu)$ the Alexander polynomial of $L$.
Similar to
Theorem \ref{t:main'},
 let $t^* = \prod_{i=1}^{\mu^*} (t^*_i)^{\lambda*_i}$.
One has
 $$
 \Delta_{L} (t'_1,\dots,t'_{\mu'},t''_1,\dots,t''_{\mu''})= \Delta_{K' \cup L'}(t'',t'_1,\dots,t'_{\mu'}) \cdot \Delta_{K''\cup L''}(t',t''_1,\dots,t''_{\mu''}),
 $$
 unless  $\mu'=0$ (ie. $L'=K'$ is a knot) and $\lambda''=0$, in which case
 $$
 \Delta_{L} (t''_1,\dots,t''_{\mu''})= \Delta_{L'' \setminus K''}(t''_1,\dots,t''_{\mu''}).
 $$
 Note that this formula were refined by Cimasoni \cite{C} for the Conway potential function.
Moreover, in relation with the signature of a colored link, one may consider the \emph{nullity}, related to the rank of the twisted first homology of the link complement. This nullity is also additive under the splice operation, in the suitable sense. Detailed statements can be found in \cite{DFL}.
\end{rmk}

\begin{ex}
Consider two copies $K' \cup L'$ and $K'' \cup L''$ of the (1,1)-colored generalized Hopf link $H_{1,2}$, see Section \ref{hopf}, where $K'$ and $K''$ are the single
components. Then, $L=L'\cup L''=H_{2,2}$ is a (1,1)-colored link, and for $\omega\in\CT^1\smallsetminus\{\pm 1\}$, we show by using C-complexes that
 $$\sigma_{L}(\omega, \omega)= \sigma_{K' \cup L'}(\omega^2,\omega) +
 \sigma_{K'' \cup L''}(\omega^2,\omega) + \delta_{(2)}(\omega) \delta_{(2)}(\omega)=
 0 + 0 + \delta_{(2)}(\omega) \delta_{(2)}(\omega).$$
This illustrates trivially that a defect appears.
\end{ex}

\begin{ex}\label{ex.referee}
For the reader convenience we add the following example. Notice the use of the formula in Theorem~\ref{t:main'} when $\omega_{i}=1$ (cf. Remark~\ref{r:one}). Let $K'\cup L'$ be the (2,4)-torus link and $K''\cup L''$ be the (4,2)-cable over the unknot with the core retained (cf. Section~\ref{s:torus}). Then, the splice of these two links along the components $K'$ and $K''$ is the (3,6)-torus link, which we shall denote $L$.

\begin{figure}\centering
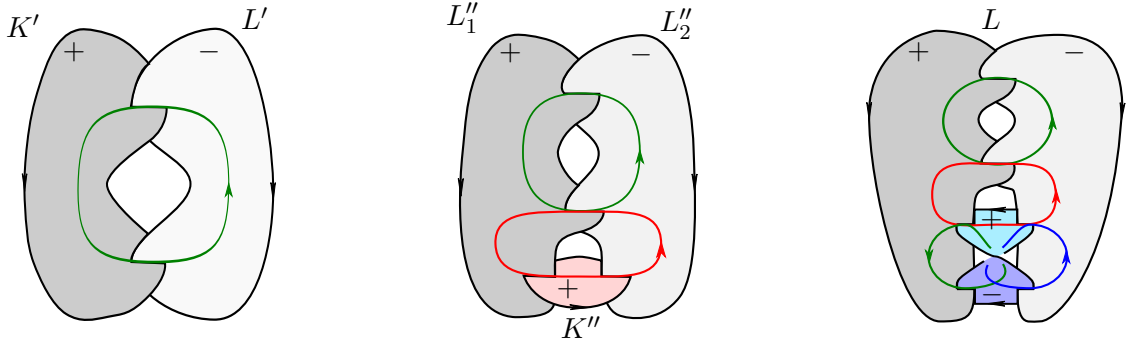
\caption{The leftmost link is the (2,4)-torus link, depicted as the boundary of a C-complex with rank 1 first homology. In the middle, the (4,2)-cable over the unknot with the core retained, bounding a rank 2 C-complex. The last diagram is the splice of the two preceding ones along $K'$ and $K''$. It represents the (3,6)-torus link.}
\label{f:complex}
\end{figure}

In the notation of Theorem~\ref{t:main'}, we
have $\lambda'=2$ and $\lambda''=(1,1)$.
For the $C$-complexes bounded by these three links one can take those
depicted in Figure~\ref{f:complex}.
To simplify the resulting Hermitian matrices~$H$, we re-denote by
$t_0,t_1,\ldots$ their arguments (in the order listed) and,
for an index set~$I$, introduce the shortcut
$\pi_I:=1+\prod_{i\in I}(-t_i)$. Then
\begin{gather*}
   H_{K'\cup L'}(\xi',\omega')=-\bar\pi_0\bar\pi_1\pi_{01},\\
   H_{K''\cup L''}(\xi'',\omega_{1}'',\omega_{2}'')=
     \bar\pi_0\bar\pi_1\bar\pi_2
     \begin{pmatrix}
       -\pi_0\pi_{12} & t_1t_2\pi_0 \\
       \pi_0 & -\pi_{012}
     \end{pmatrix},\\
   H_{L'\cup L''}(\omega',\omega_{1}'',\omega_{2}'')=
     \bar\pi_0\bar\pi_1\bar\pi_2
     \begin{pmatrix}
       -\pi_0\pi_{12} & t_1t_2\pi_0 & 0 & 0 \\
       \pi_0 & -\pi_{012} & t_0t_2\pi_1 & t_0\pi_2 \\
       0 & \pi_1 & -\pi_1\pi_{02} & -t_0\pi_1\pi_2 \\
       0 & t_1\pi_2 & \pi_1\pi_2 & -\pi_2\pi_{01}
     \end{pmatrix},
\end{gather*}
so that, up to units and factors of the form $\pi_i$, $i=0,1,\ldots$, the Alexander polynomials are
\begin{equation*}
   \Delta_{K'\cup L'}=\pi_{01},\qquad
   \Delta_{K''\cup L''}=t_0t_1^2t_2^2-1,\qquad
   \Delta_{L'\cup L''}=\pi_{012}(t_0t_1t_2+1)^2.
\end{equation*}
The computation of the signature of these matrices is straightforward: on the
respective \emph{open} tori, they are the piecewise constant functions given
by the following tables:

\vskip2mm
\begin{center}
\begin{tabular}{| c| ccccc| }
 \hline
$\Log\xi'+\Log\omega'$ &  & $1/2$ &  & $3/2$ &   \\
 \hline 
$\sigma_{K'\cup L'}(\xi',\omega')$  & $1$ & $0$ & $-1$ & $0$ & $1$   \\ 
 \hline
 \end{tabular}
\end{center}

\begin{center}
\begin{tabular}{| c| ccccccccc| }
 \hline
$\Log\xi''+2\Log\omega''$ &  & $1$ &  & $2$ & & $3$ & & $4$ &   \\
 \hline 
$\sigma_{K''\cup L''}(\xi'',\omega'')$  & $2$ & $1$ & $0$ & $-1$ & $-2$ & $-1$ & $0$ & $1$ & $2$   \\ 
 \hline
 \end{tabular}
\end{center}

\begin{center}
\begin{tabular}{| c| ccccccccc| }
 \hline
$\Log\omega'+\Log\omega''$ &  & $1/2$ &  & $1$ & & $2$ & & $5/2$ &   \\
 \hline 
$\sigma_{L'\cup L''}(\omega',\omega'')$  & $4$ & $2$ & $0$ & $-1$ & $-2$ & $-1$ & $0$ & $2$ & $4$   \\ 
 \hline
 \end{tabular}
\end{center}
\vskip2mm
\noindent
Note, however, that $L'$ is the unknot
and $L''$ is homeomorphic to $K'\cup L'$; hence,
\begin{equation*}
  \sigma_{K'\cup L'}(1,\omega')=0,\qquad
  \sigma_{K''\cup L''}(1,\omega'')=\sigma_{K'\cup L'}(\omega_1'',\omega_2'').
\end{equation*}
Now, it is immediate that the identity
$$
\sigma_L(\omega',\omega_1'',\omega_2'')=
 \sigma_{K'\cup L'}(\omega_1''\omega_2'',\vect\omega')+
 \sigma_{K''\cup L''}(\omega^{\prime2},\omega_1'',\omega_2'')
 +\defect_{(2)}(\omega')\defect_{(1,1)}(\omega_1'',\omega_2'')
$$
given by Theorem~\ref{t:main'} holds whenever
$\omega^{\prime2}\ne1$ or $\omega_1''\omega_2''\ne1$.
(It suffices to compare the values at all
triples of 8-th roots of unity.) If $\omega^{\prime2}=\omega_1''\omega_2''=1$,
we obtain an
extra discrepancy of~$1$; this phenomenon will be explained in~\cite{DFL}.
\end{ex}

As an immediate consequence of Theorem \ref{t:main'}, we see that
the Levine--Tristram signature of a splice cannot be expressed
in terms of
the Levine--Tristram signature
of its summands:
in general, the multivariate
extension is required.

\begin{cor} \label{tristram}
Let $L$ be the
splice of
$(1,1)$-colored links $K' \cup L'$ and $K'' \cup L''$,
and denote
$\lambda'=\lk(K',L')$ and $\lambda''=\lk(K'',L'')$.   Consider $L$ as a $1$-colored link. Then, for
 a character $\xi \in \CT^{1}$ such that
$\xi^{\operatorname{g.c.d.}(\lambda',\lambda'')}\ne1$,
one has
$$ \sigma_L(\xi)= \sigma_{K'\cup L'}(\xi^{\lambda''},\xi) +  \sigma_{K''\cup L''}(\xi^{\lambda'},\xi) - \lambda' \lambda'' + \delta_{\lambda'}(\xi)\delta_{\lambda''}(\xi) ,$$
where $\sigma_L(\xi)$ is the Levine--Tristram signature of  $L$.
\end{cor}

\begin{proof}
Consider the $2$-coloring on $L$ given by the splitting $L' \cup L''$.
We have
$\sigma_L(\xi,\xi)= \sigma_{K'\cup L'}(\xi^{\lambda''},\xi) +  \sigma_{K''\cup L''}(\xi^{\lambda'},\xi)+ \delta_{\lambda'}(\xi)\delta_{\lambda''}(\xi)$
by Theorem \ref{t:main'}.
On the other hand,
$\sigma_L(\xi)=\sigma_L(\xi,\xi) - \lk(L',L'')$, see Proposition \ref{prop:coloring}. By \cite[Proposition 1.2]{EN}, $\lk(L',L'')=\lambda' \lambda''$.
\end{proof}

 Theorem \ref{t:main'} is proved in Section~\ref{proof:main}.
In the special case $L'=\varnothing$, it takes the
following stronger form (we do not require that $\upsilon''\ne1$);
it is proved in Section~\ref{proof:L=0}.

\begin{add}\label{add:L=0}
Let $L\subset\Ss$ be the splice of a $(1,0)$-colored link
$K'\subset\Ss'$ and a $(1,\mu'')$-colored link $K''\cup L''\subset\Ss''$,
and let $\lambda'':=\vlk(K'',L'')$.
Then, for any character $\vect\omega\in\CT^{\mu''}$, one has
$$
\sigma_L(\vect\omega)=\sigma_{K'}
\big(\omega^{\lambda''}\big)
 +\sigma_{L''}(\vect\omega).
$$
\end{add}

\begin{rmk}
The assumption $(\upsilon',\upsilon'')\neq (1,1)$ in Theorem~\ref{t:main'} is
essential. If $\upsilon'=\upsilon''=1$, the expression for the signature
acquires an extra correction term, which can be proved to take values in $[\![ -2, 2 ]\!]$.
In many cases, this term can be computed algorithmically, and simple examples
show that typically it does not vanish. Indeed, consider two copies of the Whitehead link $K' \cup L'$ and $K'' \cup L''$. If $\omega=e^{i \pi /3}$, then $\sigma_{L}(\omega,\omega)= -1$, but $\sigma_{K' \cup L'}(1,\omega)+\sigma_{K'' \cup L''}(1,\omega)+ \delta(1)=0$ and there is a non-zero extra term.
(Addendum~\ref{add:L=0} states that the extra term does vanish whenever one
of the links~$L'$, $L''$ is empty.)
The general computation of this extra term, related to \emph{linkage} invariants
(see, e.g., \cite{b:Mu}), is addressed in a forthcoming paper \cite{DFL}.
\end{rmk}

\begin{rmk}
We
expect that the conclusion of Theorem~\ref{t:main'} would still hold without
the assumption that the characters should be rational. In fact, all ingredients of
the proof would work once recast to the language of local systems, and the main
difficulty is the very definition of the signature in homology spheres,
 where the link does not need to bound a surface and the approach
of~\cite{CF} does not apply.
(If
all links are in~$S^3$, an alternative proof can be given in terms of
$C$-complexes.) This issue will also be addressed in \cite{DFL}.
\end{rmk}

\subsection{The generalized Hopf link}
\label{hopf}
A \emph{generalized Hopf link} is the link $H_{m,n}\subset S^3$ obtained from
the ordinary positive Hopf link $H_{1,1}=V\cup U$ by replacing
its components~$V$
and~$U$ with, respectively, $m$ and $n$ parallel copies. This link is
naturally $(m+n)$-colored; its signature, which plays a special role in
the paper is given by Theorem~\ref{t:hopf}
below.
Observe the similarity to the correction term in Theorem~\ref{t:main'};
a posteriori, Theorem~\ref{t:hopf} can be interpreted as a special
case of Theorem~\ref{t:main'}, using
 the identity $\sigma_{H_{1,n}}\equiv0$
(which is easily proved independently)
and the fact
that $H_{m,n}$ is the splice of $H_{1,m}$ and $H_{1,n}$.
However, the Hopf links and their signatures are used
essentially
in the proof of Theorem~\ref{t:main'}.

\begin{thm}\label{t:hopf}
For any character
$(\vect v,\vect u)\in\CT^m\times\CT^n$, one has
$\sigma_{H_{m,n}}(\vect v,\vect u )=\defect(\vect v)\defect(\vect u)$.
\end{thm}

Certainly, Theorem~\ref{t:hopf} computes as well the signature of a
generalized Hopf link equipped with an arbitrary coloring and orientation of
components. First, one can recolor the link by assigning a separate color to
each component (cf.\ Proposition~\ref{prop:coloring} below).
Then,
one can reverse
the orientation of each negative component~$L_i$;
obviously, this operation corresponds to the substitution
$\omega_i\mapsto\bar\omega_i$.
For example, the orientation
of the original link
can be described in terms of a pair of
vectors,
viz. the linking vector $\nu\in\{\pm1\}^m$ of
the $V$-part of $H_{m,n}$ with the $U$-component of the original Hopf
link~$H_{1,1}$ and the linking vector $\lambda\in\{\pm1\}^n$ of the $U$-part with the
$V$-component. Then, {\em assuming that any two linked components of~$H_{m,n}$ are
given distinct colors}, we have
\begin{equation}\label{eq.Hopf.generalized}
  \sigma_{H_{m,n}}(\vect v, \vect u)=\delta_\nu(\vect v)\delta_\lambda(\vect u).
\end{equation}

For future references, we state a few simple properties of the defect
function $\defect$ and, hence, of the signature $\sigma_{H_{m,n}}$.
All proofs are immediate.

\begin{lemma}\label{lem:defect}
The defect function $\defect\colon\CT^\mu\to\Z$ has the following properties:
\begin{enumerate}

\item\label{defect:0}
$\defect(1)=0$;
$\defect\equiv0$ if $\mu=0$ or~$1$;
\item\label{defect:-}
$\defect(\bar\omega)=-\defect(\omega)$ for all $\omega\in\CT^\mu$;
\item\label{defect:sym}
$\defect$ is preserved by the coordinatewise action of the symmetric
group~$S_\mu$;
\item\label{defect:one}
$\defect$ commutes with the coordinate embeddings $\CT^\mu\into\CT^{\mu+1}$,
$\omega\mapsto(\omega,1)$;
\item\label{defect:conj}
$\defect$ commutes with the embeddings $\CT^\mu\into\CT^{\mu+2}$,
$\omega\mapsto(\omega,\eta,\bar\eta)$ for any
$\eta\in\CT^1$.
\end{enumerate}
\end{lemma}

\subsection{Satellite knots}\label{s:satellite}

As was first observed in \cite{EN}, the splice operation generalizes many classical
 link constructions: connected sum, disjoint union and satellites among
 others.

Our first application is
Litherland's formula for the Levine--Tristram signature of a satellite knot, which
 is a particular case of Addendum~\ref{add:L=0}.

Recall that an embedding of a
 solid torus in~$S^3$ into another solid torus in another copy of~$S^3$ is
 called
 \emph{faithful} if the image of a canonical longitude
 of the first solid torus is a canonical longitude of the second one.
 Let $V$ be an unknotted solid torus in $S^3$, and let~$k$ be a knot in the
interior of~$V$, with algebraic winding number $q$, i.e., $[k]$ is $q$
times the class of the core in $H_1(V)$.
Given any knot $K\subset S^{3}$,
the \emph{satellite knot} $K^*$ is defined as the image $f(k)$ under a
faithful embedding $f:V \rightarrow S^3$ sending the core of $V$ to $K$.

The isotopy class $K^*$ depends of course on the embedding $f$ (and even its concordance class, see \cite{Li2}). Nevertheless, its Levine--Tristram signature is determined by the
 signatures of the constituent knots and the winding number:

\begin{thm}[cf.~{\cite[Theorem~2]{Li1}}]\label{cor:satellite}
In the notation above,
the Levine--Tristram signatures of $k$, $K$ and $K^*$ are related via
$$
\sigma_{K^*}(\omega) = \sigma_K(\omega^q) + \sigma_k(\omega),\quad
 \omega\in\CT^1.
$$
\end{thm}

\begin{proof}
Let $C$ be the core of the solid torus $S^3\smallsetminus V$. The satellite
$K^*$ can be written as the splice of $K\cup\varnothing$ and $C\cup k$.
By Addendum~\ref{add:L=0}, we have
$$
\sigma_{K^*}(\omega) = \sigma_K(\omega^\lambda) + \sigma_{k}(\omega).
$$
where $\lambda:=\lk(C,k)$.
By assumption, $\lk(C,k)=q$, and the statement follows.
\end{proof}

\subsection{Iterated torus links}\label{s:torus}

Our next
 application
is another special case of Theorem~\ref{t:main'},
which provides an inductive formula for the signatures of
iterated torus links.
In particular, this class of links contains the algebraic ones,
i.e., the
links of isolated singularities of complex curves in
$\C^{2}$. Note that partial results on the equivariant signatures of the monodromy were obtained by Neumann \cite{N}. 

  Iterated torus links are obtained from an unknot by a sequence of cabling operations (and maybe, reversing
 the orientation
of some of the components).
In order to define the cabling operations (we follow the exposition in  \cite{EN}),
consider two coprime
integers $p$ and $q$ (in particular, if one of them is $0$, the other is
$\pm1$),
a positive integer $d$, a $(1,\mu')$-colored link $K'\cup L'\subset S^{3}$,
 and a small tubular neighbourhood $T'$ of $K'$ disjoint from $L'$.
Let $m,l$ be the meridian and longitude of $K'$, and $K'(p,q)$ be the
 oriented simple closed curve in $\partial T'$ homologous to   $pl+qm$.
More generally, let $d K'(p,q)$ be the disjoint union of  $d$ parallel copies of
 $K'(p,q)$ in $\partial T'$.
We say that the  link
$L= L' \cup dK'(p,q) - K'$ (resp. $L=L' \cup dK'(p,q)$) is obtained from $K'\cup L'$ by a \emph{$(dp,dq)$-cabling with the
core removed} (resp. \emph{retained}).

Let $H_{1,1}=V \cup U$ be the ordinary Hopf link.
The link  $V \cup dU(p,q)$
can be regarded as either $(1,d)$-colored or $(1,1)$-colored.
We denote the corresponding multivariate and bivariate signature functions by
$\storus_{dp,dq}$ and $\btorus_{dp,dq}$, respectively.
Note that, by Proposition~\ref{prop:coloring} below,
$$
\btorus_{dp,dq}(\vect v,\vect u)=\storus_{dp,dq}(\vect v,\vect u,\ldots,\vect u)
 -\tfrac12d(d-1)pq.
$$
In the case of core-removing,
the link~$L$ obtained by the cabling
is nothing
but the splice of $K'\cup L'$ and $V \cup dU(p,q)$.
(Similarly, in the core-retaining case, $L$ is the splice of $K' \cup L'$ and
 $V \cup U \cup dU(p,q)$.)
Hence, the following statement is an immediate consequence of
Theorem~\ref{t:main'}.

\begin{thm}\label{cor:torus}
Let $L$ be obtained from a $(1,\mu')$-colored link $K'\cup L'$ by a
$(dp,dq)$-cabling with the core removed.
For a character $\omega:=(\omega',\omega'')\in\CT^{\mu'}\times\CT^d$, let
$$
\vect\lambda':=\vlk(K',L'),\quad
\vect\lambda'':=(p,\dots,p)\in\Z^{d},\quad\text{and}\quad
\upsilon^*:=(\omega^*)^{\lambda^*},\
 \text{$*=\prime$ or $\prime\prime$}.
$$
Then, assuming that $(\upsilon',\upsilon'')\ne(1,1)$, one has
$$
\sigma_{L}(\vect\omega)=
 \sigma_{K'\cup L'}(\upsilon'',\vect\omega') +
 \storus_{dp,dq}(\upsilon',\vect\omega'')
 +\defect_{\lambda'}(\vect\omega')\defect_{\lambda''}(\vect\omega'').
$$
\end{thm}

With the evident modifications, this last corollary can be adapted to give a
formula for a $(dp,dq)$-cabling with the core retained.

The
Levine--Tristram signature of the torus link $U(p,q)$
(which coincides with $\btorus_{p,q}(1,\zeta)$ in our notation)
was computed by Hirzebruch.
  For the reader's convenience,
we cite this
result in the next lemma.
Unfortunately, we do not know
any more general statement.
	
	\begin{lemma}[see \cite{B}]
	Let $M=\{1,\ldots,p-1\}\times \{1,\ldots,q-1\} $
and let $0 < \theta \leq \frac{1}{2}$.
Consider
\begin{gather*}
 a = \# \{(i,j)\in M \,|\, \theta <
(i/p)+(j/q)
 < \theta+1 \},\\
 n = \# \{(i,j)\in M\, |\,
(i/p)+(j/q)
 =\theta \text{ or }
(i/p)+(j/q)
 =\theta+1\},\\
 b = |M|-a-n.
\end{gather*}
Then one has
$\btorus_{p,q}(1,\zeta)=b-a$ for
$\zeta= \exp(2i \pi \theta)$.

	\end{lemma}

\subsection{Multivariate vs.\ univariate signature}\label{s.univariate}

 The last application is the computation of
the
multivariate signature of a link in
terms of the Levine--Tristram signature of an auxiliary link.
(One obvious application is the case where the latter auxiliary link is
algebraic, so that its Seifert form can be computed in terms of the
 variation map $H_1(F,\partial F)\to H_1(F)$ in the homology of its Milnor fiber~$F$, see \cite{AGV}.)
 This result is similar to   \cite[Theorem 6.22]{F} by the second author and is related to the computation of signature invariants of $3$-manifolds by Gilmer,
 see  \cite[Theorem 3.6]{G}.

Let $L=L_1\cup\ldots\cup L_\mu$ be a $\mu$-colored link. For simplicity, we
assume that
 the coloring is maximal, i.e., each component of~$L$ is given
a separate color.
Let $[\lambda_{ij}]$ be the linking matrix of~$L$, i.e.,
$\lambda_{ij}=\lk(L_i,L_j)$ for $i\ne j$ and $\lambda_{ii}=0$.

Consider a character $\omega\in\CT^\mu$ and assume that
$\omega_i=\xi^{n_i}$, where $\xi:=\exp(2\pi i/n)$, for some integers $n>0$
and $0<n_i<n$.
(In particular, all $\omega_i\ne1$.)
For $i=1,\ldots,\mu$, denote
\begin{itemize}
\item
$\tlambda_i:=\sum_{j=1}^\mu n_j\lambda_{ij}$,
the weighted linking number of~$L_i$ and $L\smallsetminus L_i$;
\item
$\upsilon_i:=  \prod_{j=1}^\mu\omega_j^{\smash{\lambda_{ij}}}
 =\xi^{\smash{\tlambda_{i}}}$, where $\lambda_i$ is the
 $i$-th row of $[\lambda_{ij}]$.

\end{itemize}

Fix an integral vector
$p:=(p_1,\ldots,p_\mu)\in\Z^\mu$ and consider the \emph{monochrome} link
$\bar L:=\bar L_p(\omega)$ obtained from~$L$ by the $(n_i,n_ip_i)$-cabling
along the component~$L_i$ for each $i=1,\ldots,\mu$.
In other words, each component $L_i$ of~$L$ is regarded
$n_i$-fold, and it is replaced with $n_i$ ``simple'' components,
possibly linked (if $p_i\ne0$).

\begin{thm}\label{th.many-one}
In the notation above, one has the identity
$$
\sigma_L(\omega)=
 \sigma_{\bar L}(\xi)
 -\sum_{i=1}^\mu\btorus_{n_i,n_ip_i}(\upsilon_i,\xi)
 +\sum_{i=1}^\mu(n_i-1)\ind(\tlambda_i/n)
 +\sum_{1\le i<j\le\mu}\lambda_{ij}.
$$
\end{thm}

\begin{cor}\label{cor:p=0}
If $p=0$, the second term in Theorem~\ref{th.many-one} vanishes and one has
$$
\sigma_L(\omega)=
 \sigma_{\bar L}(\xi)
 +\sum_{i=1}^\mu(n_i-1)\ind(\tlambda_i/n)
 +\sum_{1\le i<j\le\mu}\lambda_{ij}.
$$
For small values of~$\mu$, this identity simplifies even further:
\begin{enumerate}
\item\label{mu=1}
if $\mu=1$, then
$\sigma_L(\omega)=\sigma_{\bar L}(\xi)$;
\item\label{mu=2}
if $\mu=2$ and $|\lambda_{12}|\le1$, then
$\sigma_L(\omega)=\sigma_{\bar L}(\xi)+(n_1+n_2-1)\lambda_{12}$.
\end{enumerate}
\end{cor}

\begin{proof}[Proof of Corollary~\ref{cor:p=0}]
If $p_i=0$, then  $V \cup U(n_i,0)=H_{1,n_i}$ is a generalized Hopf link;
its signature vanishes due to Theorem~\ref{t:hopf}
and Lemma~\ref{lem:defect}\iref{defect:0}.
The only other statement that needs proof is item~\ref{mu=2}, where
we have $\ind(\lambda_{12}n_i/n)=\lambda_{12}$
whenever $|\lambda_{12}|\le1$ and $0<n_i<n$, $i=1,2$.
\end{proof}

\begin{ex}
Let $L=H_{1,1}$ be the ordinary Hopf link, so that $\sigma_L\equiv0$ by
Theorem~\ref{t:hopf}
and Lemma~\ref{lem:defect}\iref{defect:0}.
On the other hand, taking $p=0$, we obtain $\bar L=H_{n_1,n_2}$;
by Theorem~\ref{t:hopf} and Proposition \ref{prop:coloring},
we get $\sigma_{\bar L}(\xi)=(1-n_1)(1-n_2)-n_1n_2$,
which agrees with Corollary~\ref{cor:p=0}\iref{mu=2}.
\end{ex}

\begin{proof}[Proof of Theorem~\ref{th.many-one}]
Denote $L[0]:=L$ and, for $i=1,\ldots,\mu$, let~$L[i]$ be the link
obtained from $L[i-1]$ by the $(n_i,n_ip_i)$-cabling along
the component~$L_i$. Each link $L[i]$ is naturally $\mu$-colored; we assign
to this link the character
$\omega[i]:=(\xi,\ldots,\xi,\omega_{i+1},\ldots,\omega_\mu)$.
In this notation, $\bar L$ is the monochrome version of $L[\mu]$ and,
by Proposition~\ref{prop:coloring},
\begin{equation}
\sigma_{\bar L}(\xi)=\sigma_{L[\mu]}(\omega[\mu])
 -\sum_{1\le i<j\le\mu}n_in_j\lambda_{ij}.
\label{eq.uni}
\end{equation}
Introduce the following characters:
\begin{itemize}
\item
$\tomega'[i]:=(\xi,\ldots,\xi)\in\CT^{n_i}$;
\item
$\tomega''[i]$, obtained from $\omega$ by
replacing each
$\omega_j$ with $n_j|\lambda_{ij}|$
copies of $\xi^{\smash{\sg\lambda_{ij}}}$, if $j\le i$, or
$|\lambda_{ij}|$ copies of $\omega_j^{\smash{\sg\lambda_{ij}}}$, if $j>i$;
\item
$\tomega[i]$, obtained from $\omega$ by
replacing each $\omega_j$ with
$n_j|\lambda_{ij}|$ copies of $\xi^{\smash{\sg\lambda_{ij}}}$.
\end{itemize}
 By definition, $L[i]$ is the splice of $L[i-1]$ and {$V \cup {n_i}U(1,p_i)$.
Then Theorem~\ref{t:main'}  applies and, for each $i=1,\ldots,\mu$,
\begin{equation}
\sigma_{L[i]}(\omega[i])=
 \sigma_{L[i-1]}(\omega[i-1])
 +\btorus_{n_i,n_ip_i}(\upsilon_i,\xi)
 +\defect(\tomega'[i])\defect(\tomega''[i]).
\label{eq.step}
\end{equation}
We have $\Log\tomega'[i]=n_i/n$; since $0<n_i<n$, this implies
\begin{equation}
\defect(\tomega'[i])=1-n_i.
\label{eq.prime}
\end{equation}
  One can show that
$\defect(\tomega''[i])=\defect(\tomega[i])-\sum_{j=i+1}^{\mu}(1-n_j)\lambda_{ij}$.
Indeed, $\tomega[i]$ is obtained from $\tomega''[i]$ by
$|\lambda_{ij}|$ operations of
replacement of
a single copy of $\omega_j^{\smash{\sg\lambda_{ij}}}$
with
copies of $\xi^{\smash{\sg\lambda_{ij}}}$
for all $j>i$; as in~\eqref{eq.prime},
one such operation increases the value of~$\defect$ by $(1-n_j)\sg \lambda_{ij}$.
The character $\tomega[i]$ has all entries equal to~$\xi$ or $\bar\xi$,
with the exponent sum equal to $\tlambda_i$.
Using Lemma~\ref{lem:defect}\iref{defect:conj} and~\iref{defect:sym} to
cancel the pairs $\xi$, $\bar\xi$, we get
$\defect(\tomega[i])=\ind(\tlambda_i/n)-\tlambda_i$;
hence,
\begin{equation}
\defect(\tomega''[i])=\ind(\tlambda_i/n)
 -\sum_{j=1}^{i-1} n_j\lambda_{ij}
 -\sum_{j=i+1}^{\mu}\lambda_{ij}.
\label{eq.2prime}
\end{equation}
Applying~\eqref{eq.step} inductively and taking into account~\eqref{eq.prime}
and~\eqref{eq.2prime}, we arrive at
$$
\sigma_{L[\mu]}(\omega[\mu])=
 \sigma_L(\omega)
 +\sum_{i=1}^\mu\btorus_{n_i,n_ip_i}(\upsilon_i,\xi)
 -\sum_{i=1}^\mu(n_i-1)\ind(\tlambda_i/n)
 +\sum_{1\le i<j\le\mu}(n_in_j-1)\lambda_{ij},
$$
and the statement of the theorem follows from~\eqref{eq.uni}.}
\end{proof}

\section{Signature of a link in a homology sphere}\label{s:prel}

In the early sixties Trotter introduced a numerical knot invariant called the
signature \cite{b:Tr}, which was subsequently extended to links by Murasugi
\cite{b:Mu}. This invariant was generalized to a function
 (defined via
Seifert forms) on
$S^1\subset\Cc$ by Levine and Tristram \cite{b:Tri,b:Le}.
 It was then
reinterpreted in terms of  coverings and intersection forms of $4$-manifolds
by Viro \cite{Viro,V2}. Our definition of the signature of a colored link follows
Viro's approach and the $G$-signature theorem, see also \cite{GLM, F}.

\subsection{Twisted signature and additivity}\label{s:twisted}

We start with recalling the definition and some properties of the twisted
signature of a $4$-manifold.

Let $N$ be a compact smooth oriented $4$-manifold   with boundary and $G$ a finite abelian
group.
Fix a covering $N^G\to N$, possibly
ramified, with
$G$ the group of deck transformations.
If the covering is ramified, we assume that the ramification locus~$F$ is a
union of smooth compact surfaces $F_i\subset N$ such that
\begin{enumerate}
\item\label{F.1}
$\partial F_i=F_i\cap\partial N$;
\item\label{F.2}
each surface~$F_i$ is transversal to~$\partial N$, and
\item\label{F.3}
distinct surfaces intersect transversally, at double points,
and away from~$\partial N$.
\end{enumerate}
Items~\iref{F.1} and~\iref{F.2} above mean that
each component $F_i$ of $F$ is a properly embedded surface. For short,
a compact surface $F\subset N$ satisfying all conditions
\iref{F.1}--\iref{F.3} will be called \emph{properly immersed}.
Under these assumptions, $N^G$ is an oriented rational homology manifold and
we have a well-defined Hermitian intersection form
$$
\iif\colon H_2(N^G;\Cc)\otimes H_2(N^G;\Cc)\to\Cc.
$$
Regard the homology groups $H_*(N^G;\Cc)$ as $\Cc[G]$-modules
and consider
the form
$$
\varphi\colon H_2(N^G;\Cc)\otimes H_2(N^G;\Cc)\to\Cc[G],\quad
 \varphi(x,y):=\sum_{g \in G}\<x,gy\>g.
$$
Since $G$ is
abelian,
this form is sesquilinear, i.e.,
$\varphi(g_{1}x,g_{2}y)=g_{1}g_{2}^{-1}\varphi(x,y)$ for all
$g_{1},g_{2}\in G$.

Any multiplicative character $\chi\colon G\to\Cc^*$ induces a homomorphism
$\Cc[G]\to\Cc$ of algebras with involution ($zg\mapsto\bar zg\1$ in $\C[G]$ is
mapped to $\eta\mapsto\bar\eta$ in~$\C$).
This makes~$\C$ a
$\C[G]$-module,
and we can consider the \emph{twisted homology}
$$
H^\chi_*(N,F):=H_*(N^G;\C)\otimes_{\C[G]}\C.
$$
In this notation, the ramification locus~$F$ is omitted whenever it is empty
or understood.
The form~$\varphi$ above induces a $\C$-valued Hermitian form~$\varphi^\chi$
on $H^\chi_2(N,F)$; explicitly, the latter is given by
$$
\varphi^\chi(x\otimes z_{1},y\otimes z_{2})
 =z_{1}\bar z_{2}\sum_{g\in G}\<x,gy\>\chi(g).
$$

We will denote by $\sign(N)$ the ordinary signature of the
$4$-manifold $N$,
i.e., that of the
form $\iif$ on $H_{2}(N)$.
The \emph{twisted signature}, denoted by $\sign^\chi(N,F)$, is the signature
of the above Hermitian form~$\varphi^\chi$.

\begin{rmk}\label{rem:G}
One can easily see that the twisted homology $H_*^\chi(N,F)$ and twisted
signature $\sign^\chi(N,F)$ are independent of the group~$G$ used in the
construction: they only depend on the pair $(N,F)$ and the multiplicative
character $\chi\colon H_1(N\smallsetminus F)\to\C^*$, which must be assumed
of finite order. In particular, we can always take for~$G$ the ``smallest'' cyclic
group, viz.\ the image of~$\chi$.
Indeed, there is an obvious canonical isomorphism between $H_*^\chi(N,F)$ and
the $\chi$-equitypical summand
$$
V_*^\chi(G):=
 \bigl\{x\in H_*(N^G;\Cc)\bigm|gx=\chi(g)x\text{ for all $g\in G$}\bigr\},
$$
and the form $\varphi^\chi$ is
$|G|$-times the restriction to
$V^\chi(G)$ of the ordinary intersection index form $\iif$.
Now, if $G$ is replaced with a larger group $G'\onto G$, the transfer
map induces an
isomorphism $V_*^\chi(G)\to V_*^\chi(G')$, multiplying the intersection index form
by another positive factor $[G':G]$; hence, the signature is preserved.
\end{rmk}

Of particular
interest
is
the behavior of the signature under the gluing of
manifolds. Recall that, by Novikov's additivity, if $N_{1}$ and $N_{2}$ are two
$4$-manifolds such that $\partial N_{1}=-\partial N_{2}$ and
$N=N_{1}\cup_{\partial}N_{2}$, then the ordinary and the twisted signatures
of $N$ satisfy
$$
\sign(N)=\sign(N_{1})+\sign(N_{2})\quad\text{and}\quad
 \sign^\chi(N,F)=\sign^\chi(N_1,F_1)+\sign^\chi(N_2,F_2).
$$
Of course, in
the twisted version we assume that the ramification loci $F_1$ and $F_2$
match along the boundary, $F=F_1\cup_\partial F_2$, and the characters
on~$N_1$, $N_2$
are the restrictions of
a character on~$N$.
If $N_{1}$ and $N_{2}$ are glued along a part of their boundaries only, the
above equalities may fail. This situation was completely studied by Wall in
\cite{b:Wa}. For our purposes we only need a particular case of Wall's
theorem, which we state below. The result is given in terms of ordinary
signatures, but, as mentioned by Wall at the end of his paper, the same
conclusion holds if we consider twisted signatures.
\begin{thm}[see~\cite{b:Wa}] \label{additivity}
Suppose that $\partial N_1 \simeq M_1 \cup M_0$ and
$\partial N_2\simeq M_2\cup -M_0$, where $M_0,M_1$ and $M_2$ are
$3$-manifolds glued along their
common boundary. Let $N:=N_1\cup_{M_0}N_2$ and $X:=\partial M_0= \partial
M_1= \partial M_2$. Consider the $\Cc$-vector spaces
$A_i:=\Ker[H_1(X;\C)\rightarrow H_1(M_i;\C)]$,
$i=0,1,2$, and let
$$
K(A_0,A_1,A_2):=\frac{A_0\cap(A_1 + A_2)}{(A_0\cap A_1)+(A_0\cap A_2)}.
$$
If $K(A_0,A_1,A_2)$ is trivial, then we have $\sign(N)=\sign(N_{1})+\sign(N_{2}).$
\end{thm}
\begin{rmk}\label{r:wall}
Note that the additivity in Theorem~\ref{additivity} holds if at least two of
$A_0,A_1,A_2$ are equal. Moreover, Wall shows in his article that the vector
space $K(A_0,A_1,A_2)$ is independent of the order of the $A_{i}$'s. When
working with twisted signatures, we shall use the notation
$A_i^\chi:=\Ker[H_1^\chi(X)\to H_1^\chi(M_i)]$,
$i=0,1,2$.
\end{rmk}

\subsection{The signature of a link}\label{s:defs}

Let $L$ be a $\mu$-colored link in an integral homology sphere $\Ss$.
By Alexander duality,
the group
$H_1(\Ss\smallsetminus L)$ is generated by the meridians of the components of
$L$. We shall denote by  $m_{i}^k$ the meridians of the components of
the sublink $L_i$ of $L$ of color $i=1,\ldots,\mu$.

Let $\Zz^\mu$ be the free multiplicative group generated by $t_1,\dots,t_\mu$.
The coloring on $L$
gives rise to a homomorphism
$c:H_1(\Ss\smallsetminus L) \to \Zz^\mu$,
$m_i^k\mapsto t_i$, $i=1,\ldots,\mu$.
We consider multiplicative characters
$H_1(\Ss\smallsetminus L)\to\C^*$
 that \emph{respect the coloring}, i.e.,
 factor
 through~$c$.
They are determined by their values on the
 generators~$t_i$,
and the group of
such characters
can be identified with $\CT^{\mu}$. Through this identification,  the character $\omega \in \CT^\mu$ assigns the meridians of the components of the sublink $L_i$ to $\omega_i$.
With a certain abuse of the language, we will shortly speak about the character $\omega$ on $L$ and say that $\omega$ assigns $\omega_i$
 to (each component of) $L_i$.

The next proposition asserts that
$\omega\colon H_1(\Ss\smallsetminus L)\to\C^*$ extends to a finite order character
$\omega\colon H_1(N\smallsetminus F)\to\C^*$ (also denoted by the same
letter~$\omega$),
where $N$ is a $4$-manifold
bounded by~$\Ss$ and $F\subset N$ is a certain
properly
immersed
surface.

\begin{propo} \label{p:bounding}
Let $L$ be a $\mu$-colored link in an integral homology sphere~$\Ss$.
Then, there exists a compact smooth oriented $4$-manifold $N$ and
 an oriented properly immersed
surface $F=F_1\cup\ldots\cup F_\mu$ in $N$
such that
\begin{itemize}
\item
$\partial N=\Ss$ and $\partial F_i=L_i$ for $i=1,\ldots,\mu$,
\item
the group
$H_1(N\smallsetminus F) \simeq \ZZ^\mu$ is
freely
generated
by the meridians $\bar{m}_i$ of~$F_i$, and
\item
one has $[F_i,\partial F_i]=0$ in $H_2(N,\partial N)$.
\end{itemize}
As a consequence, any character $\omega\in\CT^\mu$
extends to a unique character
$$
\omega\colon H_1(N \smallsetminus F)\to\C^*,\quad
 \bar m_{i}\mapsto\omega_i.
$$
\end{propo}

 For short, as for characters on links, we will speak about the character $\omega$ on $F$ and say that $\omega$ assigns $\omega_i$
 to the component $F_i$.

We postpone the proof of this statement till Section~\ref{proof:bounding}.

Now, we are ready to define the main object of
study in this paper.

\begin{de}\label{sign}
The \emph{signature} of a $\mu$-colored link $L\subset\Ss$ is the map
$$
\begin{array}{ccll}
\sigma_L\colon  & \CT^\mu & \longrightarrow &\Zz \\
 & \vect\omega & \longmapsto & \sign^{\vect\omega}(N,F) - \sign(N),
\end{array}
$$
where $N$ and~$F$ are as in Proposition~\ref{p:bounding}.
\end{de}

The signature of a $\mu$-colored link in $\Ss$ is related to
invariants previously
defined by Gilmer \cite{G}, Smolinski \cite{S}, Levine \cite{Le}
and the first author \cite{F}. The interested reader can find detailed
history in \cite{CF}. In the case
where $\Ss=S^3$,
the
signature considered in this paper coincides with the signature defined by
Cimasoni--Florens \cite{CF} for $\omega\in\CT$ with
$\omega_i\neq1$ for all $i=1, \dots, \mu$.
In our present work we shall deal also with the
case $\omega_{i}=1$. The following remark should be clear from the definition
of the signature of a colored link.

\begin{rmk}\label{r:one}
Let $L$ be a $\mu$-colored link in $\Ss$,
and let $\vect\omega \in \CT^\mu$ be a vector such that $\omega_{i}=1$.
Then, the following equality
holds:
$$
\sigma_L(\dots, 1, \dots)=
\sigma_{L_1\cup\ldots\cup\hat{L}_i\cup\ldots\cup L_{\mu}}(\dots,\hat1,\dots).
$$
\end{rmk}

Another important observation is the fact
that the coloring of the link is essential: it is not enough to
merely assign a value of a character to each component of the link.
More precisely,
we have the following relation (whose proof for $S^3$ found in~\cite{CF}
extends to integral homology spheres almost literally:
the
extra term is due to the perturbation of the union $F_\mu\cup F_{\mu+1}$ of two
components of the ramification locus into a single surface).

\begin{propo}[{see~\cite[Proposition~2.5]{CF}}]\label{prop:coloring}
Let $L:=L_1\cup\ldots\cup L_{\mu+1}$ be a $(\mu+1)$-colored link, and
consider the $\mu$-colored link $L':=L'_1\cup\ldots\cup L'_\mu$ defined via
$L'_i=L_i$ for $i<\mu$ and $L'_\mu=L_\mu\cup L_{\mu+1}$. Then, for any
character $\vect\omega\in\CT^\mu$, one has
$$
\sigma_{L'}(\vect\omega)=
 \sigma_L(\omega_1,\ldots,\omega_\mu,\omega_\mu)-\lk(L_\mu,L_{\mu+1}).
$$
\end{propo}

\begin{cor}\label{c:coloring}
The multivariate signature of a generalized Hopf link $H_{m,n}$
does not depend on the coloring, provided that
linked components are given distinct colors.
\end{cor}

In particular,
Proposition \ref{prop:coloring}
provides a relation between
the restriction
of the multivariate signature of a colored link
to the diagonal in $\CT^\mu$ and
the Levine--Tristram signature of the underlying monochrome link.

As asserted in the following proposition, the signature of a colored link is
well defined,
i.e.,
independent of the pair $(N,F)$ chosen to compute
it. This is a consequence of Novikov's additivity and the $G$-signature
theorem.

\begin{propo} \label{invariance}
For all $\vect\omega \in \CT^{\mu}$,
the signature of $(\Ss,L)$ at $\vect\omega$
$$ \sigma_L(\vect\omega)=  \sign^{\vect\omega}(N,F) - \sign(N)   $$
does not depend on the pair $(N,F)$.
\end{propo}
\begin{proof}
Given two pairs $(N',F')$ and $(N'',F'')$ as in Proposition~\ref{p:bounding},
consider $W:=N'\cup_\partial{-N''}$ and $F:=F'\cup_\partial{-F''}\subset W$.
By Novikov's additivity, the statement of the proposition would follow if we
show that $\sign^\omega(W,F)=\sign W$.

To compute the twisted signature, we can use the group
$G:=C_{q_1}\times\dots\times C_{q_\mu}$,
where $q_i$ is the order of $\omega_i$, $i=1,\dots,\mu$, see
Remark~\ref{rem:G}. Crucial is the fact that, under the assumptions on
$(W,F)$, this group results in a smooth closed manifold~$W^G$.

Consider the equitypical decomposition of the $\C[G]$-module
\begin{equation}
H:=H_2(W^G;\C)=\bigoplus_\rho V^\rho,
\label{eq:V-rho}
\end{equation}
where $\rho$ runs over all multiplicative characters $G\to\C^*$.
Since the intersection index form $\iif$ is $G$-invariant, this decomposition
is orthogonal. Denote by $\sign V^\rho$ the signature of the restriction of
the form to~$V^\rho$. By Remark~\ref{rem:G}, we have
$\sign^\omega(W,F)=\sign V^\omega$.

The argument below is a slight generalization of~\cite{Ro}
(see also \cite[Lemma 2.1]{CG}).

Each space $V^\rho$ can further be decomposed (not canonically) into the
orthogonal sum of two subspaces $V^\rho_+$ and $V^\rho_-$ with, respectively,
positive and negative definite restriction of $\iif$.
Summation over all characters gives us a $G$-invariant decomposition
$H=H_+\oplus H_-$. Recall that the \emph{$G$-signature} of an element
$g\in G$ is
$\sign(g,W):=\trace g_*|_{H_+}-\trace g_*|_{H_-}\in\C$.
It is well defined; in fact, using~\eqref{eq:V-rho}, we have
$$
\sign(g,W)=\sum_\rho\rho(g)\sign V^\rho.
$$
Multiplying this by $\bar\omega(g)$ and summing up over all $g\in G$,
we arrive at
$$
|G|\sign V^\omega=\sum_{g\in G}\bar\omega(g)\sign(g,W)
 =|G|\sign V^1+\sum_{g\ne1}(\bar\omega(g)-1)\sign(g,W).
$$
(Recall that irreducible characters are orthogonal. For the second equality,
we use the identity $\sum_g\sign(g,W)=|G|\sign V^1$, $g\in G$, which is
the first equality with $\omega\equiv1$.) By the usual transfer argument,
$\sign V^1=\sign W$. Summarizing, we conclude that
\begin{equation}
\sign^\omega(W,F)-\sign(W)=\frac1{|G|}\sum_{g\ne1}(\bar\omega(g)-1)\sign(g,W)
\label{eq:s-s}
\end{equation}
is a linear combination of the $g$-signatures $\sign(g,W)$ with $g\in G$ and
$g\ne1$.

Since $W$ is a smooth manifold, we can use the $G$-signature
theorem~\cite{AS,Go}, which expresses the $g$-signature $\sign(g,W)$ in terms of
the fixed point set of~$g$.
We use repeatedly the fact that each surface~$F_i$ is connected and the
covering is ''uniform'' along~$F_i$; hence, the extra factor appearing in the
$G$-signature theorem depends on the element $g\in G$ only and does not
depend on a particular component of the fixed point set.

If $1\ne g\in C_{q_i}$ lies in one of the factors of~$G$, its fixed point set
is $F_i$ and $\sign(g,W)$ is a multiple of $[F_i]^2$.
By Proposition~\ref{p:bounding},
$[F_i^*,\partial F_i^*]=0\in H_2(N^*,\partial N^*)$ for $*=\prime$ or
$\prime\prime$; hence, $[F_i]=0$ and
$\sign(g,W)=0$.

If $g\in C_{q_i}\times C_{q_j}$ lies in the product of two factors (but not
in either of them), the fixed point set is $F_i\cap F_j$ and $\sign(g,W)$ is
a multiple of $\<[F_i],[F_j]\>=0$ (since, as above,
$[F_i]=[F_j]=0$).

In all other cases, the fixed point set is empty (there are no triple
intersections); hence, $\sign(g,W)=0$.
Summarizing, $\sign(g,W)=0$ whenever $g\ne1$;
in view of~\eqref{eq:s-s},
this implies that $\sign^\omega(W,F)=\sign W$ and concludes the proof.
\end{proof}

\subsection{Proof of Proposition~\ref{p:bounding}}\label{proof:bounding}

Consider an integral surgery presentation for $\Ss$ given by a framed oriented link $T=T_{1}\cup\ldots\cup T_{k}$ in $S^{3}$. Since $\Ss$ is a homology sphere, we may assume that $T$ is algebraically split and that the surgery coefficients of each of its components are $\pm 1$ \cite[Theorem A]{b:Mat}. The link $L$ can be represented by a collection of curves in $S^{3}\smallsetminus T$.

Let $N$ be the $4$-manifold obtained
by attaching $2$-handles to $B^4$ along
the components of
$T$ according to their framings. Since the linking matrix of $T$ is diagonal
with $\pm 1$ entries, we may slide the knots in $L$ over the attached handles
in order to obtain a presentation of $L$ in $\Ss$ such that
$\lk_{\Ss}(L_{i},T_{j})=0$ for all $i=1,\dots,\mu$ and $j=1,\dots,k$. Since
all $L_{i}$ are disjoint from the attaching tori of the handles, we can
consider a surface $F$ in $B^4$, the $0$-handle of $N$, such that $F$ is a
union of compact connected oriented smooth surfaces $F_1,\dots,F_\mu$, and
each $F_i$ is smoothly embedded with $\partial F_i=L_i$.

We have the following commutative diagram:
$$
\begin{CD}
0=H^{1}(N)@>>>H^1(N\smallsetminus F)@>>>
 H^2(N,N\smallsetminus F)@>>>H^2(N)\\
@.@.@VVV@VVV\\
@.@.H_{2}(F,\partial F)@>i_*>>H_{2}(N,\partial N),
\end{CD}
$$
where, by Alexander and Lefschetz duality, the two
vertical arrows are
isomorphisms and the
inclusion homomorphism~$i_*$ is trivial, as
$\lk_{\Ss}(L_{i},T_{j})=0$ and thus
$[F_{i},\partial F_{i}]=0\in H_{2}(N,\partial N)$ for all $i=1,\dots,\mu$.
It follows that
$H^{1}(N\smallsetminus F)$ is
canonically
isomorphic to
$H_{2}(F,\partial F)=\Z^{\mu}$,
and
the latter
group is
freely
generated by
the fundamental classes $[F_i,\partial F_i]$.
Repeating the same computation over the finite field $\Bbb F_p$, we get
$H^1(F\smallsetminus F;\Bbb F_p)=H_{2}(F,\partial F;\Bbb F_p)$ and, since the
dimension of this vector space does not depend on~$p$, we conclude that
the homology group
$H_1(F\smallsetminus F)=\operatorname{Hom}(H_1(F\smallsetminus F),\Zz)$ is
freely generated by the elements of the dual basis, i.e.,
the meridians $\bar m_{i}$ of the components $F_i\subset N$.
\qed

\section{Proof of Theorem~\ref{t:main'}}\label{s:proofs}

\subsection{The auxiliary Hopf link}\label{s:associated}

In
the proof of Theorem~\ref{t:main'} it will be useful to have some control
over the surface $F$ used to compute the colored signatures;
namely,
sometimes we want the distinguished component~$K$ to bound a disk.
The proof
of the following lemma is a straightforward adaptation of the proof of
Proposition~\ref{p:bounding}.

\begin{lemma}\label{l:disc}
Let $K \cup L$ be a $(1,\mu)$-colored link in $\Ss$. Then, the pair $(N,F)$
in Proposition \ref{p:bounding} can be chosen of the form $(N,D\cup F)$, where $D$ is a disk, $K=\partial D$ and $L_{i}=\partial F_{i}$.
\end{lemma}
\begin{proof}
As explained in the proof of Proposition~\ref{p:bounding}, we can consider an integral surgery presentation for $\Ss$ given by a framed oriented link $T=T_{1}\cup\ldots\cup T_{k}$ in $S^{3}$, where $T$ is algebraically split and the surgery coefficients of each of its components are $\pm 1$. Moreover, the link $K\cup L$ can be represented by a collection of curves in $S^{3}\smallsetminus T$ such that $\lk(K,T_{i})=\lk(L_{j},T_{i})=0$ for all $i,j$.

Notice that we can obtain $K\cup L\subset\Ss$ by starting with $U\cup L\subset S^{3}\smallsetminus T$, where $U$ is the unknot, and performing surgery on unknotted curves $C_{1},\dots, C_{t}$ in $S^{3}\smallsetminus (T\cup U\cup L)$ with framings $\epsilon_{i}=\pm 1$ to do some crossing changes on $U$ to obtain $K$. It is clear that we might assume $\lk(U,T_{i})=0$ for all $i$ and that the curves $C_{i}$ may be chosen such that $\lk(C_{i},C_{j})=0$ if $i\neq j$ and $\lk(C_{i},T_{j})=\lk(C_{i},U)=\lk(C_{i},L_{j})=0$ for all $i$ and $j$.

 The link $U\cup L$ in $S^{3}$ bounds a
 properly immersed surface $D\cup F_1\cup\ldots$ in $B^4$.
Indeed, one has
\begin{enumerate}
\item
$L_{i}=\partial F_i=F_i\cap\partial B^{4}$ and $U=\partial D=D\cap\partial B^{4}$;
\item
$D$ and each surface~$F_i$ are transversal to~$\partial B^{4}$, and
\item
distinct surfaces intersect transversally, at double points,
and away from~$\partial B^{4}$.
\end{enumerate}

Consider the 4--manifold $N$ obtained by attaching 2--handles to $B^{4}$
along the components of $T\cup C_{1}\cup\ldots\cup C_{t}$ according to their
framings. By construction we obtain the link $K\cup L$ sitting in
$\Ss=\partial N$ and bounding $F$. Moreover, the above conditions on the
linking numbers guarantee that the proof of Proposition~\ref{p:bounding}
follows word by word with the manifold $N$ and the surface $F$ considered in
this proof.
\end{proof}

Let
$(N,D\cup F)$ be the pair constructed in Lemma~\ref{l:disc} and fix a tubular
neighborhood $B\cong D \times B^2$ of $D$ in $N$, see
Figure~\ref{fig:special_case}.
Without loss of generality, by taking~$B$ small enough,
we may assume that, \emph{up to orientation of the components},
the pair $(B, (D \cup F) \cap B)$ has boundary $(S^3,H_{1,m})$, where $m$ is
the number of points in $D \cap F$. The components of
$H_{1,m}=V\cup U_1\cup\ldots\cup U_m$
inherit an orientation from $D\cup F$,
and we color them
according to the decomposition $D \cup F_1\cup\ldots\cup F_\mu$.

Assume that the original link is given a character $(\vect v,\vect u)$.
This character extends to $D\cup F$ and restricts to a character, also
denoted by $(\vect v,\vect u)$, on $H_{1,m}$.
Occasionally, we will replace~$D$ with several parallel copies,
obtaining a link $H_{n,m}$, and change the
character on the $V$-part of $H_{n,m}$, while keeping~$\vect u$ on the
$U$-part.
We always assume that linked components are given distinct colors,
but we allow a nonstandard orientation of the $V$ part, describing it
by a linking vector $\nu$, cf.\ the paragraph prior
to~\eqref{eq.Hopf.generalized}.

\begin{lemma}\label{lem.new}
Consider the link $H_{n,m}=V\cup U$
equipped with the coloring, orientation, and character $\vect u$
on the $U$-part as explained above.
  Then, for any character~$\vect v$ on the $V$-part and any
  linking vector~$\nu$, one has
$$
\sigma_{H_{n,m}}(\vect v,\vect u)=-\delta_\nu(\vect v)\delta_\lambda(\vect u),
$$
where $\lambda:=\vlk(K,L)$.
\end{lemma}

\begin{proof}
The $U$-part of the link can be described as follows:
for each $i=1,\ldots,\mu$, there is a number of components, all carrying the
same color and character~$\omega_i$, oriented in a random way but so that
the entries of the linking vector
(with respect to a fixed positive component of the $V$-part) sum up
to~$\lambda_i$. (These components correspond to the geometric intersection
points, and their orientation reflects the sign of the intersection.)
Hence, the statement is an immediate
consequence of~\eqref{eq.Hopf.generalized} and
the definition of~$\delta$, as the copies of $\pm\Log u_i$ would sum up
to $\lambda_i\Log u_i$.
\end{proof}

%




%

\subsection{A special case}\label{s:special}
The
next lemma is straightforward; it is stated for references.
We will
use it to apply Wall's Theorem \ref{additivity}. Certainly, the statement on $H_0^\chi(X)$
extends to any topological space~$X$, whereas that on $H_1^\chi(X)$ extends
to any space with abelian fundamental group.

\begin{lemma}\label{lem:torus}
Let $X\cong T^2$ be a $2$-torus and $\chi\colon H_1(X)\to\Cc^*$ a multiplicative
character. Then $H_1^\chi(X)=H_1(X;\Cc)$, $H_0^\chi(X)=H_0(X;\Cc)$ if $\chi\equiv1$
and $H_1^\chi(X)=H_0^\chi(X)=0$ otherwise.
\end{lemma}

We start by proving a special case of Theorem~\ref{t:main'}, which will be
useful later on and whose proof contains the key ingredients used
to establish the general formula. In the following lemma we study the effect on the colored signatures
of changing the component $K$ of a $(1,\mu)$-colored
link $K\cup L\subset\Ss$ to a collection of $\nu$
parallel curves, i.e.\ of performing a $(\nu,0)$-cabling.
This operation is equivalent to the splice of
$K\cup L\subset\Ss$ and $H_{1,\nu}\subset S^3$.

Let $\KK \cup L$ be the resulting $(\nu+\mu)$-colored link.
Denote $\lambda:=\vlk(K,L)$
and, for a character
$\vect\omega\in\CT^{\mu}$, let
$\upsilon:=\omega^{\lambda}$.
For a character $\vect\zeta\in\CT^\nu$, let
 $\pi:=\prod_{i=1}^\nu\zeta_i$.

\begin{lemma}\label{cable}
In the notation above, assuming that $(\upsilon,\pi)\ne(1,1)$, one has
$$
\sigma_{\KK\cup L}(\vect\zeta,\vect\omega)=
 \sigma_{K\cup L}(\pi,\vect\omega)
-\defect(\zeta)\defect_\lambda(\omega).
$$
\end{lemma}

\begin{figure}\centering
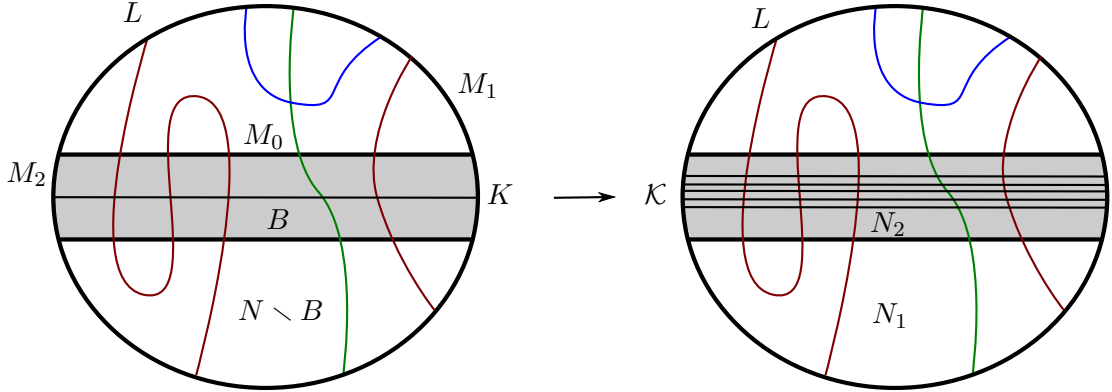
\caption{This diagram represents the pairs $(N,D\cup F)$ and $(N,\DD\cup F)$. The gray band is the four ball $N_{2}=B=D\times B^{2}$.}
\label{fig:special_case}
\end{figure}

\begin{proof}
 The diagram in Figure~\ref{fig:special_case} might
 help one
follow the construction.
Let $(N,D\cup F)$ be the pair constructed in Lemma~\ref{l:disc} for the link
$K \cup L \subset\mathbb S=\partial N$  and fix a tubular
neighborhood $B\cong D \times B^2$ of $D$ in $N$.
The pair $(N,D \cup F)$ can be written as the union
$$
(N\smallsetminus B, F\cap (N\smallsetminus B)) \cup (B, (D \cup F) \cap B)
$$
glued along $(D \times S^1, F \cap \partial B)$.
As explained in Section~\ref{s:associated},
the boundary of $(B, (D \cup F) \cap B)$ is
$(S^3,H_{1,m})$, where $H_{1,m}$ inherits orientations of the components and coloring with the
associated character $(\pi,\omega)$.


We use Wall's Theorem~\ref{additivity} to relate the twisted and non-twisted
signatures of $(N,D \cup F)$ and
$(N\smallsetminus B,F\cap (N\smallsetminus B))$.
To this end, define $N_1=N \smallsetminus B$,
$M_1=\Ss\smallsetminus\iTN(K)$, $N_{2}=B$, $M_{2}=\TN(K)$ and
$M_0= D\times S^1$, where $\TN$ stands for a small tubular neighborhood and
$\iTN$ is its interior. One has
$\partial N_1= M_1 \cup M_0$ and
$\partial N_2=M_2\cup{-M_0}$,
and in both cases the manifolds are glued along
$X:=\partial D\times S^1=K\times S^{1}$.
Let $m$ and $\ell$ be the meridian and longitude
of $K$, which generate $H_{1}(X)$. Following the notation of
Theorem~\ref{additivity}, we have $A_0= A_1=\langle\ell\rangle$ and
$A_2=\langle m\rangle$, which implies that $K(A_0,A_1,A_2) =0$ and
thus
\begin{equation}\label{e:down}
\sign(N)=\sign(N_1 \cup N_2) =\sign(N_{1})+\sign(N_{2})=\sign(N_{1})
\end{equation}
since $N_2$ is contractible.

We now make the corresponding computation with twisted coefficients.

Let $\vrho:=(\pi,\vect\omega)$; we will use the same notation for the
extensions of~$\vrho$ to the other spaces involved.
We need to study the relationship between
$\sign^\vrho(N_1\cup N_2,D\cup F)$ and the twisted signatures of~$N_{1}$
and~$N_{2}$.
The group
$H_1(X)$ is generated by $m,\ell$ and,
since $\rho(m)=\upsilon$ and $\rho(\ell)=\pi$ are not both trivial,
we have $H_1^\vrho(X)=0$, see Lemma~\ref{lem:torus}.
This
trivially implies
$ K\bigl(A^\vrho_0, A^\vrho_1,A^\vrho_2\bigr)=0$,
and Wall's Theorem~\ref{additivity} yields
\begin{equation}\label{e:twist}
\sign^\vrho(N_1\cup N_2,D\cup F)=
 \sign^\vrho(N_{1},F\cap N_{1})+
 \sign^\vrho
 (N_{2},(D\cup F)\cap N_{2}).
\end{equation}

Since the boundary of $(N_{2},(D\cup F)\cap N_{2})$ is
$(S^3,H_{1,m})$, by Lemmas~\ref{lem.new} and \ref{lem:defect}\iref{defect:0}
we have
$$
\sign^\vrho
 {(\pi,\omega)}(N_{2},(D\cup F)\cap N_{2})-\sign(N_{2})=
 \sigma_{H_{1,m}}(\pi,\omega)=\delta(1) \delta_\lambda(\omega)=0.
$$
%
Combining equations~\eqref{e:down}
and~\eqref{e:twist} with Definition~\ref{sign} of signature, we get
\begin{equation}\label{e:first}
\sigma_{K\cup L}(\vrho)=\sign^\vrho(N,D\cup F)-\sign(N)=
 \sign^\vrho(N_1,F\cap N_{1})-\sign(N_1).
\end{equation}

Now, consider the link $\KK\cup L$. We can assume that $\KK$ lies in the
tubular neighborhood $M_{2}=\partial D \times B^2$ of $K$ in $\Ss$. The link
$\KK$ bounds a collection of $\nu$ parallel disks $\DD\subset N_{2}=B$ and
the pair $(N_2,(\DD\cup F)\cap N_{2})$ has boundary $(S^3, H_{\nu,m})$, with the generalized Hopf link $H_{\nu,m}$ carrying the character
$(\zeta,\omega)$ and corresponding orientations.
Similar to~\eqref{e:twist}, one has
%
\begin{equation}\label{e:twistmult}
\sign^{(\vect\zeta,\vect\omega)}(N_1\cup N_2,\DD\cup F)
 =\sign^\vrho(N_1, F\cap N_{1})
 +\sign^{(\vect\zeta,\omega)}(N_2, (\DD\cup F)\cap N_{2}).
\end{equation}
Moreover, we can compute the signature of $H_{\nu,m}$
from the pair $(B,(\DD\cup F)\cap B)$;
thus,
 since $N_2$ is contractible,
\begin{equation}\label{e:hopfm}
\sign^{(\vect\zeta,\omega)}(N_2, (\DD\cup F)\cap N_{2})=
-\defect(\zeta)\defect_\lambda(\omega),
\end{equation}
see  Lemma~\ref{lem.new}.
Using the pair $(N,\DD\cup F)$ to compute the signature of $\KK\cup L$, we have
\begin{align*}
\sigma_{\KK\cup L} (\vect\zeta,\vect\omega)
&\underset{\hphantom{\eqref{e:down},\eqref{e:twistmult}}}=
 \sign^{(\vect\zeta,\vect\omega)}(N,\DD \cup F) - \sign(N)\\
&\underset{\eqref{e:down},\eqref{e:twistmult}}{=}
 \sign^\vrho(N_1, F\cap N_{1}) + \sign^{(\vect\zeta,\omega)}(N_2, (F\cup\DD)\cap N_{2})-\sign(N_{1})\\
&\underset{\eqref{e:first},\eqref{e:hopfm}}{=}
 \sigma_{K\cup L}(\vrho)
-\defect(\zeta)\defect_\lambda(\omega).\qedhere
\end{align*}
%
\end{proof}

\begin{figure}\centering
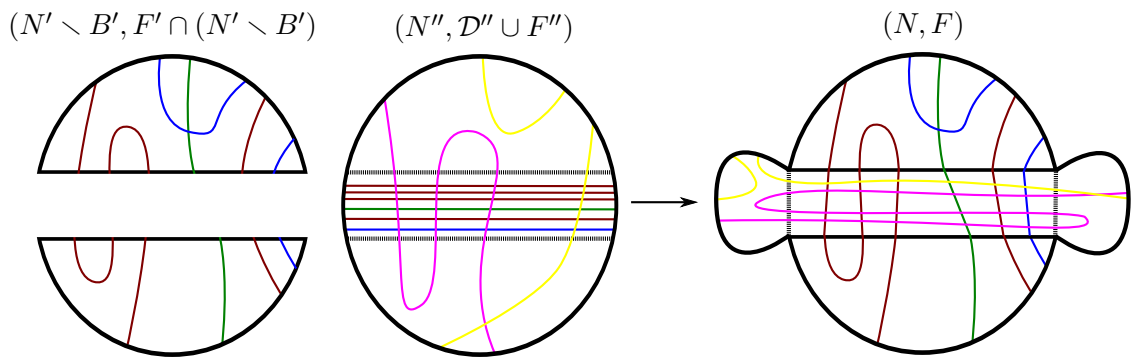
\caption{The third diagram represents the pair $(N,F)$ used to compute the signature of the splice of $K'\cup L'$ and $K''\cup L''$. This pair is obtained identifying parts of the boundary of  $(N'\smallsetminus B', F' \cap ( N'\smallsetminus B'))$ and $(N'',\DD''\cup F''$).}
\label{fig:main}
\end{figure}

\subsection{Proof of Theorem~\ref{t:main'}}\label{proof:main}

 The diagram in Figure~\ref{fig:main} might be useful to follow the details.
Let $(N',D'\cup F')$ be the pair
constructed in Lemma~\ref{l:disc} for the link $K'\cup L'\subset\Ss'$
and fix a small tubular neighborhood
$B'\cong D'\times B^2$
of $D'$ in $N'$. Since $(\upsilon',\upsilon'')\neq (1,1)$, we can repeat the arguments
in
the proof of
Lemma~\ref{cable} involving Wall's theorem to obtain
\begin{equation}\label{e:complement}
\sigma_{K'\cup L'}(\upsilon'',\vect\omega')= \sign^{(\upsilon'',\vect\omega')}
 (N'\smallsetminus B', F' \cap ( N'\smallsetminus B')) - \sign(N'\smallsetminus B').
\end{equation}
%

By
construction, the
surface $(D'\cup F')\cap B'$ consists of
the disk $D'$ and a union of $m'$ parallel disks transversal to $D'$
(those
coming from $F'$).
Consider now a pair $(N'',D''\cup F'')$ given by
Lemma~\ref{l:disc} for $K''\cup L''\subset\Ss''$.
Replace $K''$ with $m'$ parallel copies,
{
(with the orientations coherent with the signs
of the intersection points of $D'$ and $F'$) to obtain a
$(m'+\mu'')$-colored link $\KK''\cup L''\subset\Ss''$,
to which we}
assign the character
$(\omega',\vect\omega'')\colon H_1(\Ss \smallsetminus (\KK'' \cup L'')) \to \Cc^*$.
In a similar way,
replace the disk $D''$ with $m'$ parallel
copies to obtain a pair $(N'',\DD''\cup F'')$.
We may assume that
the disks constituting~$\DD''$ lie in a small neighborhood
$B''\cong D''\times B^{2}$,
and we color the components of $\DD''$ in
{ accordance with the colors of}
the $m'$ parallel disks coming from the surface $F'$ in $(D'\cup F')\cap B'$.

In the boundary of $B''$, we obtain a generalized Hopf link $H_{m',m''}$ (up to
orientation of the components, cf.\ Section~\ref{s:associated})
carrying the character $(\omega',\omega'')$.
Lemma~\ref{cable} applied to the
$(m',0)$-cabling of $K''\cup L''$ along~$K''$ yields
\begin{multline}\label{e:copies}
\sigma_{\KK''\cup L''}(\omega',\vect\omega'')
 =\sign^{(\omega',\vect\omega'')}(N'',\DD''\cup F'')-\sign(N'')\\
 =\sigma_{K''\cup L''}(\upsilon',\vect\omega'')
   -\delta_{\lambda'}(\omega')\delta_{\lambda''}(\omega'').
 \end{multline}
%
{(For the last term, Lemma~\ref{lem.new} is applied twice, first
to~$\omega'$, then to~$\omega''$.)}
Now, let us look at the pair $(N,F)$ obtained as the gluing
\begin{equation}\label{e:glue}
(N'\smallsetminus B', F' \cap ( N'\smallsetminus B'))\cup(N'',\DD''\cup F''),
\end{equation}
%
with the solid torus $\TN'=D'\times\partial B^{2}$ in the boundary of
$B'$
identified with the solid torus
$\TN''=\partial D''\times B^{2}$,
which is a tubular neighborhood $\TN(K'')$ of $K''$ in
$\Ss''$. The identification is made in such a way that the disk
$D'\subset\TN'$ is glued to $B^{2}\subset\TN''$. Moreover, the
$m'$ disks removed from the surfaces $F'$ in the intersection
$F' \cap ( N'\smallsetminus B')$ are filled with the corresponding
$m'$ disks constituting~$\DD''$. Notice that the boundary of
$(N,F)$ is nothing but $(\Ss,L)$, i.e., the splice in question.
Furthermore, by the construction, the pair
$(N,F)$ can be used to compute the colored signature of $(\Ss,L)$,
that is,
$$
\sigma_{L}(\vect\omega',\vect\omega'')=\sign^{(\vect\omega',\vect\omega'')}(N,F)-\sign(N).
$$
To complete
the proof we shall study the behavior of the twisted and classical
signatures of $N$ with respect to the decomposition \eqref{e:glue}. By
Theorem~\ref{additivity}, the signatures will be additive with respect to this
decomposition if at least two of the kernels $A_{0},A_{1},A_{2}$ coincide in
the classical and in the twisted version. In the classical version, we are
dealing with the group $H_{1}(\partial\TN')$, generated by
$m_{K'}=\ell_{K''}$ and $\ell_{K'}=m_{K''}$, the meridian and longitude of
$K'$ and $K''$ which are identified in \eqref{e:glue}. It is clear that the
kernels of the inclusion of $H_{1}(\partial\TN')$ into both
$H_{1}(\Ss'\smallsetminus\inte\TN(K'))$ and $H_{1}(\TN(K''))$ are generated
by $\ell_{K'}=m_{K''}$, and thus, by Wall's theorem we have
\begin{equation}\label{e:d}
\sign(N)=\sign(N'\smallsetminus B')+\sign(N'').
\end{equation}
In the twisted version,
the space
$H_{1}^{(\vect\omega',\vect\omega'')}(\partial\TN')
 =H_{1}^{(\upsilon',\upsilon'')}(\partial\TN')$
vanishes due to Lemma~\ref{lem:torus} and the assumption
$(\upsilon',\upsilon'')\neq (1,1)$. Hence,
Theorem~\ref{additivity} yields
\begin{equation}\label{e:t}
\sign^{(\vect\omega',\vect\omega'')}(N,F)=\sign^{(\upsilon'',\vect\omega')}(N'\smallsetminus B', F' \cap ( N'\smallsetminus B'))+\sign^{(\omega',\vect\omega'')}(N'',\DD''\cup F'').
\end{equation}
Putting these equations together, we obtain
\def\eqbox#1{\setbox0\hbox{$\scriptstyle\eqref{e:d},\eqref{e:t}$}%
 \hbox to\wd0{\hss$\scriptstyle#1$\hss}}
\begin{align*}
\sigma_{L}(\vect\omega',\vect\omega'')
&\mathrel{\smash{\underset{\eqbox{\eqref{e:d},\eqref{e:t}}}{=}}}
\sign^{(\upsilon'',\vect\omega')}(N'\smallsetminus B',F'\cap( N'\smallsetminus B'))
 +\sign^{(\omega',\vect\omega'')}(N'',\DD''\cup F'')\\
 &\eqbox{}\qquad{}-\sign(N'\smallsetminus B')-\sign(N'')\\
&\underset{\eqbox{\eqref{e:complement},\eqref{e:copies}}}{=}
 \sigma_{K'\cup L'}(\upsilon'',\vect\omega')
 +\sigma_{K''\cup L''}(\upsilon',\vect\omega'')
  -\delta_{\lambda'}(\omega')\delta_{\lambda''}(\omega'').
\pushQED\qed\qedhere
\end{align*}
%

%
\subsection{Proof of Addendum~\ref{add:L=0}}\label{proof:L=0}

Applying Theorem~\ref{t:main'}
to the splice of $K'$ and $K''\cup L''$, we obtain

$$
\sigma_{L}(\vect\omega)
 =\sigma_{K'}(\upsilon'')+\sigma_{K''\cup L''}(1,\vect\omega'')
 -\defect(1)\defect_{\lambda''}(\omega'')
 =\sigma_{K'}(\upsilon'')+\sigma_{L''}(\vect\omega'').
$$
Thus,
it suffices to justify that, in this particular case,
Theorem~\ref{t:main'} holds even if $\upsilon''=1$.

Let $\vrho:=(\upsilon',\upsilon'')$.
In the proof of
Theorem~\ref{t:main'},
the assumption
$\vrho\neq(1,1)$ was only used to establish that the
twisted homology group $H_{1}^\vrho(\partial\TN')$ is
trivial,
yielding~\eqref{e:t}.
If $\vrho=(1,1)$,
this group
is no longer trivial, but we
shall see that \eqref{e:t} still holds if $L'$ is empty.

By Remark~\ref{r:wall}, we only need to show that two among the three
groups
$A^\vrho_{0}$, $A^\vrho_{1}$, and $A^\vrho_{2}$ are equal.
We are dealing with the kernels
of the inclusions
$H_{1}^\vrho(\partial\TN')\to H_{1}^{\vect\omega}(M_{i})$,
where $M_{0}=\Ss'\smallsetminus\inte\TN(K')$,
$M_{1}=\Ss''\smallsetminus\inte\TN(K'')$ and $M_{2}=\TN(K'')$. Since
$\vrho=(1,1)$, the restriction of the covering to
$\partial\TN'$ is trivial and the group
$H_{1}^\vrho(\partial\TN')$ is
generated
by the lifts $\widetilde m_{K'}$ and $\tilde\ell_{K'}$ of the meridian and
longitude of $K'$, which are identified respectively with the longitude and
meridian of $K''$.
While the generators of
$A^\vrho_{1}$
are not evident,
the groups $A^\vrho_{0}$
and $A^\vrho_{2}$ are easily seen to be equal. Indeed,
since $L'$ is empty and $\upsilon''=1$, the group
$H_{1}^{\upsilon''}(\Ss'\smallsetminus\inte\TN(K'))$ is
the homology group of the trivial covering of $M_{0}$, therefore,
$\tilde\ell_{K'}=\widetilde m_{K''}$ generates
$A^\vrho_{0}$.
On the other hand, since $L'$ is empty,
$|\vect\lambda'|=0$ and we do not need to work with parallel
copies of $K''$. It follows that the group $H_{1}^{\upsilon''}(\TN(K''))$ is
the homology group of the trivial covering of $M_{2}$ and
$\widetilde m_{K''}=\tilde\ell_{K'}$ generates
$A^\vrho_{2}$. We conclude
that
$A^\vrho_{0}=A^\vrho_{2}$, completing the proof.
\qed

\section{The generalized Hopf link}\label{s:Hopf}

In this section, we compute the signature of a generalized Hopf link
using the $C$-complex approach of \cite{CF}. This approach works only for
characters with all components distinct from one. Thus, we define
the \emph{open character torus} $\torus^\mu$, obtained from $\CT^\mu$
by removing all ``coordinate planes'' of the form $\omega_i=1$,
$i=1,\ldots,\mu$.

\begin{figure}\centering
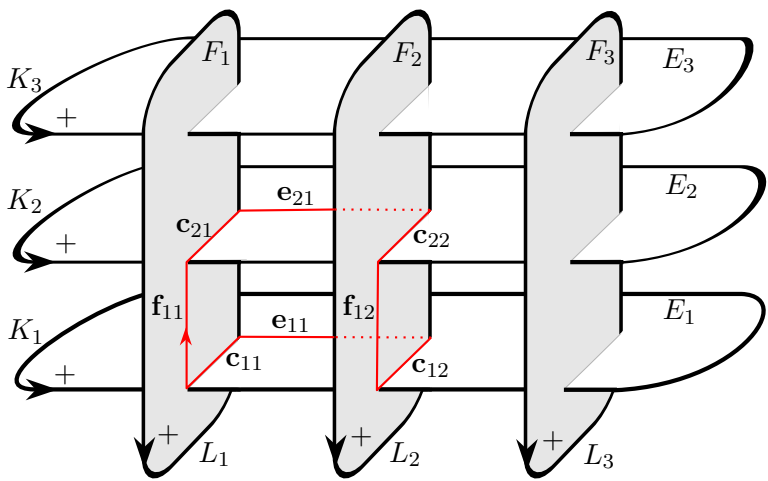
\caption{This diagram represents a generalized Hopf link of type $H_{3,3}$. The link is depicted bounding a bicolored oriented $C$-complex $S$, which is the union of the disks $E_{i}$ and $F_{j}$. The red loop is $\aa_{11}$, whose homology class $\alpha_{11}$ is an element of $H_{1}(S)$.}
\label{fig:hopf}
\end{figure}

\subsection{$C$-complexes and Seifert forms}\label{s:C-complex}

We
recall briefly the notion of $C$-complex of a $\mu$-colored link
and its application to the computation of the signature. To avoid excessive
indexation, we consider the special case of the bivariate signature of a
bicolored link; for the general case and further details, see~\cite{CF}.

Thus, let $K\cup L$ be a bicolored link, with the coloring $K\mapsto1$,
$L\mapsto2$. (We do not assume $K$ or $L$ connected.)
A \emph{$C$-complex} is a pair of
Seifert surfaces~$E$ for~$K$ and~$F$ for~$L$,
possibly disconnected, which intersect each other transversally (in the
stratified sense) and only at \emph{clasps}, i.e., smooth simple arcs, each
connecting a point of~$K$ to a point of~$L$.
(In the general case of more than two colors, the only additional requirement
is that all triple intersections of Seifert surfaces involved must be empty.)
Let $S:=E\cup F$. Then, for each
pair $\epsilon,\delta=\pm1$, one can consider the \emph{Seifert
form}
$$
\theta^{\epsilon\delta}\colon H_1(S)\otimes H_1(S)\to\Zz,
$$
defined as follows. Pick a class $\alpha\in H_1(S)$ and represent it by a
simple closed curve $\aa\subset S$ satisfying the following condition:
each clasp
$\cc\subset E\cap F$ is either disjoint from~$\aa$ or
entirely contained in~$\aa$. It is immediate that such a curve~$\aa$ can be
pushed off~$E$ in the direction~$\epsilon$ (with respect to the coorientation
of~$E$, which is part of the structure)
and off~$F$ in the
direction~$\delta$, so that the resulting curve~$\aa'$ is disjoint from~$S$.
Then, for another class $\beta\in H_1(S)$, the value
$\theta^{\epsilon\delta}(\alpha\otimes\beta)$ is the linking coefficient of
the shift~$\aa'$ and a cycle representing~$\beta$.

Now, given a pair of complex units $(\eta,\zeta)\in\torus^2$, consider the
form
\begin{equation}
H(\eta,\zeta):=(1-\bar\eta)(1-\bar\zeta)
\bigl(\theta^{1,1}-\zeta \theta^{1,-1}-\eta \theta^{-1,1}+\eta\zeta \theta^{-1,-1}\bigr).
\label{eq:H}
\end{equation}
  The extensions of $\theta^{\epsilon\delta}$ to
$H_1(S)\otimes\Cc$ are  chosen \emph{sesquilinear}; hence
this form is Hermitian and it has a well-defined
signature.
It computes the signature of $K\cup L$.

\begin{thm}[see~\cite{CF}]\label{t:CF}
The restriction to the open torus $\torus^2$ of the bivariate signature of a
bicolored link $K\cup L$ is given by
$$
\sigma_{K\cup L}\colon(\eta,\zeta)\mapsto\sign H(\eta,\zeta).
$$
\end{thm}

\begin{rmk}
Strictly
speaking, the statement of Theorem~\ref{t:CF} is the \emph{definition} of
signature in~\cite{CF}.
This
definition is equivalent to the
conventional one, see \cite[Section 6.2]{CF}.
\end{rmk}

In general, for a $\mu$-colored link~$L$, one should consider a
$\mu$-component $C$-complex~$S$ and all $2^\mu$ possible shift directions,
arriving at a Hermitian form $H(\omega)$, $\omega\in\torus^\mu$, computing
the signature $\sigma_L(\omega)$.
The \emph{nullity} $\nl_L(\omega):=\nl H(\omega)$ is also an
invariant of~$L$; it is given by the following theorem.

\begin{thm}[{see \cite[Theorem 6.1]{CF}}]\label{th.nullity}
For any character
$\omega\in\torus^\mu$ in the open character torus,
one has $\nl_L(\omega)=\dim H_1^\omega(S^3\smallsetminus L)$.
\end{thm}

\begin{propo}\label{prop.nullity}
Let $H:=H_{m,n}$ be a generalized Hopf link. Then, for any
$(\eta,\zeta)\in\torus^m\times\torus^n$, one has
$$
\nl_H(\eta,\zeta)=\begin{cases}
m+n-3,&\text{if $\Log\eta\in\Z$ and $\Log\zeta\in\Z$},\\
m-1,&\text{if $\Log\eta\notin\Z$, $\Log\zeta\in\Z$},\\
n-1,&\text{if $\Log\eta\in\Z$, $\Log\zeta\notin\Z$},\\
0,&\text{otherwise}.\\
\end{cases}
$$
\end{propo}

\begin{proof}
The generalized Hopf
link $H_{m,n}$ can be thought of as the splice of the links
$H_{1,m}$ and $H_{1,n}$.
Since obviously
$S^3\smallsetminus H_{1,m}\cong S^1\times D_m$, where $D_m$ is an
$m$-punctured disk, for any pair $(\upsilon,\eta)\in\CT^1\times\torus^m$ we
have
$$
\dim H_1^{(\upsilon,\eta)}(S^3\smallsetminus H_{1,m})=\begin{cases}
m-1,&\text{if $\upsilon=1$},\\
0,&\text{if $\upsilon\ne1$}.
\end{cases}
$$
A similar relation holds for $H_{1,n}$;
in view of Theorem~\ref{th.nullity},
the statement of the proposition
follows from the Mayer--Vietoris exact sequence,
with Lemma~\ref{lem:torus} taken into account.
\end{proof}

\subsection{Proof of Theorem \ref{t:hopf}}\label{proof:hopf}

Due to Remark~\ref{r:one}, we have
$$
\sigma_{H_{m,n}}(\ldots,1,\ldots,\vect\zeta)=
 \sigma_{H_{m-1,n}}(\ldots,\hat1,\ldots,\vect\zeta),\qquad
\sigma_{H_{m,n}}(\vect\eta,\ldots,1,\ldots)=
 \sigma_{H_{m,n-1}}(\vect\eta,\ldots,\hat1,\ldots).
$$
These formulas agree with the statement of the theorem,
see Lemma~\ref{lem:defect}\iref{defect:one},
and it suffices
to compute the restriction of
$\sigma_{H_{m,n}}$ to the
open character torus $\torus^{m+n}$.

Consider the group $G:=\Zm\times\Zn$.
We will use the cyclic indexing for the components of the link and other
related objects. Let $K_i$, $i\in\Zm$ be the first $m$ parallel components
and $L_j$, $j\in\Zn$, the last $n$ parallel components.

By an obvious semicontinuity argument, for any $\mu$-colored link~$L$, the
multivariate signature $\sigma_L(\omega)$ is constant on
each connected component of each stratum
$\{\omega\in\torus^\mu\,|\,\nl_L(\omega)=\const\}$.
If $L=H_{m,n}$,
the strata are given by Proposition~\ref{prop.nullity}: they
are the hyperplanes
$P_p\times\torus^n$ and $\torus^m\times Q_q$, where
$$
P_p:=\{\vect\eta\in\torus^m\,|\,\Log\vect\eta=p\},\quad
Q_q:=\{\vect\zeta\in\torus^n\,|\,\Log\vect\zeta=q\},\quad
p,q\in\Zz,
$$
and all pairwise intersections thereof.
It is immediate that the bi-diagonal
$\eta_1=\ldots=\eta_m$, $\zeta_1=\ldots=\zeta_n$ meets each
component of each stratum;
hence, it suffices to compute the restriction of the signature
function to this bi-diagonal.
Due to Corollary~\ref{c:coloring},
this is equivalent to computing the
\emph{bivariate} signature
$\tilde\sigma\colon\torus^2\to\Zz$,
of the \emph{bicolored} generalized Hopf link
(with the coloring $K_i\mapsto1$, $L_j\mapsto2$, $(i,j)\in G$),
and
the formula to
be established takes the form
$$
\tilde\sigma(\eta,\zeta)=
 \delta_{[m]}(\eta)\delta_{[n]}(\zeta)=
 \bigl(\ind(m\Log\eta)-m\bigr)\bigl(\ind(n\Log\zeta)-n\bigr),\qquad
 (\eta,\zeta)\in\torus^2.
$$

Consider $m$ disjoint parallel disks~$E_i$ and $n$ disjoint parallel disks
$F_j$, so that $\partial E_i=K_i$, $i\in\Zm$, and $\partial F_j=L_j$,
$j\in\Zn$. We can assume that each component~$L_j$ intersects each disk~$E_i$
at a single point $e_{ij}$, so that these points appear in~$L_j$ in the
cyclic order given by the orientation.
These points cut~$L_j$ into segments
$\ll_{ij}:=[e_{ij},e_{i+1,j}]$, $i\in\Zm$.
Likewise, each component~$K_i$
intersects each disk~$F_j$ at a single point~$f_{ij}$, the points appearing
in~$K_i$ in the cyclic order given by the orientation,
and we will speak about the segments
$\kk_{ij}:=[f_{ij},f_{i,j+1}]\subset K_i$,
$j\in\Zn$.
Finally, assume that the intersection $E_i\cap F_j$ is a segment
$\cc_{ij}:=[e_{ij},f_{ij}]$ (a clasp).
Then,
letting $E:=\bigcup_iE_i$ and $F:=\bigcup_jF_j$,
the union $S:=E\cup F$ is a bicolored $C$-complex for $H_{m,n}$,
and we can apply Theorem~\ref{t:CF}.

\begin{rmk}\label{rem:mn>1}
If $m\le1$ or $n\le1$, then $H_1(S)=0$
and
the signature is trivially zero.
Hence, from now on we can assume that $m,n\ge2$.
Note though that formally this case does agree with the statement of the
theorem, as $\defect\equiv0$ on $\CT^0$ and~$\CT^1$.
\end{rmk}

In each disk~$E_i$, consider a collection of segments (simple arcs)
$\ee_{ij}:=[e_{ij},e_{i,j+1}]$, $j\in\Zn$,
disjoint except the common boundary points
and such that their union is a circle~$C_i$ parallel to $\partial E_i=K_i$
(and the points appear in this circle
in accordance with their cyclic order).
Consider similar segments
$\ff_{ij}:=[f_{ij},f_{i+1,j}]\subset F_j$, $i\in\Zm$,
forming circles $D_j\subset F_j$ parallel to $\partial F_j=L_j$.
Then, the group $H_1(S)$ is generated by the classes~$\alpha_{ij}$
of the loops
$$
\aa_{ij}:=\cc_{ij}\cdot\ff_{ij}\cdot\cc_{i+1,j}\1\cdot\ee_{i+1,j}\cdot
 \cc_{i+1,j+1}\cdot\ff_{i,j+1}\1\cdot\cc_{i,j+1}\1\cdot\ee_{ij}\1,
$$
$(i,j)\in G$, connecting the points
$$
e_{ij}\to f_{ij}\to f_{i+1,j}\to e_{i+1,j}\to e_{i+1,j+1}\to f_{i+1,j+1}
 \to f_{i,j+1}\to e_{i,j+1}\to e_{ij}
$$
(in the order of appearance). The construction is illustrated in Figure \ref{fig:hopf}.
Note that we do not assert that these elements form a basis: they are
linearly dependent. However, we will do the computations in the free abelian
group $\HH:=\bigoplus_{i,j}\Zz\alpha_{ij}$, $(i,j)\in G$; this
change will increase the kernel of the form, but it will not affect the
signature.

The proof of the following lemma is postponed till
Section~\ref{proof:Seifert}.

\begin{lemma}\label{lem:Seifert}
Given $\epsilon,\delta=\pm1$,
the only nontrivial values taken by the Seifert form
$\theta^{\epsilon\delta}$ on the pairs of
generators~$\alpha_{ij}$
are as follows:
$$
\alpha_{ij}\otimes\alpha_{ij}\mapsto-\epsilon\delta,\quad
\alpha_{ij}\otimes\alpha_{i-\epsilon,j}\mapsto\epsilon\delta,\quad
\alpha_{ij}\otimes\alpha_{i,j+\delta}\mapsto\epsilon\delta,\quad
\alpha_{ij}\otimes\alpha_{i-\epsilon,j+\delta}\mapsto-\epsilon\delta,
$$
where $(i,j)\in G$.
\end{lemma}

Consider
the Hermitian inner product $\iif$ on
$\HH\otimes\Cc$ with respect to which $\alpha_{ij}$ is an orthonormal basis,
and use this inner product to identify operators
$A\colon\HH\otimes\Cc\to\HH\otimes\Cc$ and sesquilinear forms
$\alpha\otimes\beta\mapsto\<\alpha A,\beta\>$.
(In accordance with the contemporary right group action conventions, our
matrices act on \emph{row} vectors by the \emph{right} multiplication.)
Then, in order to complete the proof, we need to find the eigenvalues of the
self-adjoint operator $H(\eta,\zeta)$ as in~\eqref{eq:H}.
To this end, consider the unitary
representation $\rho\colon G\to U(\HH\otimes\Cc)$
given by the index shifts of
the basis elements, viz.
$$
\rho(p,q)\colon\alpha_{ij}\mapsto\alpha_{i+p,j+q},\qquad
 (p,q), (i,j)\in G.
$$
This is the regular representation of~$G$, and its equitypical summands are
all of dimension one;
letting $\xi_k:=\exp(2\pi i/k)$, the
summands are generated by the bi-eigenvectors
$$
v_{ij}:=\frac1{mn}\sum_{(r,s)\in G}\xi_m^{-r}\xi_n^{-s}\alpha_{i+r,i+s},
 \qquad(i,j)\in G,
$$
so that $v_{ij}$ is an eigenvector of $\rho(p,q)$
with the eigenvalue $\xi_m^{pi}\xi_n^{qj}$, $(p,q)\in G$.
It is immediate from Lemma~\ref{lem:Seifert} that all forms
$\theta^{\epsilon\delta}$ are $G$-invariant; hence, they all have
the same eigenvectors $v_{ij}$.
In fact, we have more: using Lemma~\ref{lem:Seifert}, one easily concludes
that
$$
\theta^{\epsilon\delta}=-\epsilon\delta\bigl(\rho(0,0)-\rho(-\epsilon,0)
 -\rho(0,\delta)+\rho(-\epsilon,\delta)\bigr).
$$
Combining this with~\eqref{eq:H} and simplifying,
we see that the spectrum of $H(\eta,\zeta)$
consists of the $mn$ real numbers
$\lambda(\eta,\xi_m^{i})\lambda(\zeta,\bar\xi_n^{j})$,
$(i,j)\in G$,
where
$$
\lambda(x,y):=i(1-\bar x)(1-\bar y)(1-xy).
$$
(Note that $\lambda(x,y)\in\RR$ whenever $|x|=|y|=1$, which we always assume.)
Thus, the signature of $H(\eta,\zeta)$ equals
$\sigma_m(\eta)\sigma_n(\zeta)$,
where $\sigma_k(x)$
stands for the ``signature'' of
the sequence of real numbers $\lambda(x,\xi_k^i)$, $i\in\ZZ/k$.

To compute $\sigma_k(x)$, we make the following simple observations (where
$x,y\in\Cc$ are complex units, $|x|=|y|=1$):
\begin{enumerate}
\item\label{i:Hopf.1}
$\lambda(x,1)=\lambda(1,x)=\lambda(x,\bar x)=0$;
\item\label{i:Hopf.2}
$\lambda(x,-1)=\lambda(-1,x)=-4\operatorname{Im}x$;
\item\label{i:Hopf.3}
with $y\ne1$ fixed, the function $\lambda(x,y)$, $x\ne1$,
changes sign only at $x=\bar y$.
\end{enumerate}
Combining items~\ref{i:Hopf.2} and~\ref{i:Hopf.3}, we see that
$\sg\lambda(x,y)=\sg(\Log x+\Log y-1)$ for all $x,y\ne1$.
Consider the function $\phi\colon(0,1)\to\ZZ$,
$t\mapsto\sigma_k(\exp(2\pi it))$.
It follows that $\phi$ is locally constant at each point $t\in(0,1)$
such that $kt\notin\ZZ$,
whereas $\phi(t+0)-\phi(t)=\phi(t)-\phi(t-0)=1$ for $kt\in\ZZ$.
Together with the normalization $\phi(\frac12)=\sigma_k(-1)=0$ given by
items~\ref{i:Hopf.1} and~\ref{i:Hopf.2} above, we have $\phi(t)=\ind(kt)-k$.
In other words,
$\sigma_k(x)=\ind(k\Log x)-k$,
which concludes the proof of the theorem.
\qed

\subsection{Proof of Lemma~\ref{lem:Seifert}}\label{proof:Seifert}
We keep the notation introduced in Section~\ref{proof:hopf}.

The cycles~$\aa_{ij}$
do satisfy the conditions imposed in
the definition of~$\theta^{\epsilon,\delta}$, see Section~\ref{s:C-complex}.
To compute the linking coefficients,
we consider the $\Pi$-shaped Seifert surface
$\Pi_{ij}$ for $\aa_{ij}$ composed of three ``squares'':
\begin{itemize}
\item
one in~$E_i$, bounded by the loop
$\cc_{ij}\cdot\kk_{ij}\cdot\cc_{i,j+1}\1\cdot\ee_{ij}\1$,
\item
one in~$E_{i+1}$, bounded by the loop
$\ee_{i+1,j}\cdot\cc_{i+1,j+1}\cdot\kk_{i+1,j}\1\cdot\cc_{i+1,j}\1$, and
\item
one bounded by
$\ff_{ij}\cdot\kk_{i+1,j}\cdot\ff_{i,j+1}\1\cdot\kk_{ij}\1$,
disjoint from~$S$ except the boundary.
\end{itemize}
Shift this surface together with the cycle, first off~$S$ in the direction
$(\epsilon,\delta)$, and then "towards the reader", in the direction of the
clasp~$\cc_{ij}$, i.e., from~$e_{ij}$ to $f_{ij}$.
Then, the intersection index of the shift~$\Pi'_{ij}$ and
another cycle $\aa_{pq}$, $(p,q)\in G$,
is easily seen geometrically; below, we give a
simple visual description of the result.

Observe
that all cycles $\aa_{ij}$ lie in the graph
$S':=\bigcup_iC_i\cup\bigcup_jD_j\cup\bigcup_{i,j}\cc_{ij}$.
Contracting each clasp to a point, we project~$S'$ to an $(m\times n)$-grid
in  the torus~$T^2$, identified with $D_0\times C_0$.
The cycle $\aa_{ij}$ projects to the boundary $\partial s_{ij}$
of the $(i,j)$-th cell~$s_{ij}$ of the grid. (This cell can be visualized as
the projection of the third square, the one not contained in~$E$,
in the description of $\Pi_{ij}$; the two other squares collapse to the two
 horizontal edges of~$s_{ij}$.)
Let $s_{ij}'\subset T^2$ be a small shift of~$s_{ij}$
off the grid in the direction $(\epsilon,\delta)$. Then one has
$$
\Pi_{ij}'\circ_{S^3}\aa_{pq}=
 -\operatorname{vert}s_{pq}\circ_{T^2}\operatorname{hor}s'_{ij},
$$
where $\operatorname{vert}s$ stands for the
sum of the two ``vertical'' edges in~$\partial s$, $\operatorname{hor}s$ stands for the sum of the two ``horizontal'' edges in the boundary of a (shifted) cell~$s$ (with their boundary orientation),
and the
second intersection index is in the torus~$T^2$, oriented so that
the projection of each cycle~$\aa_{ij}$ is the \emph{positive} boundary
of the respective cell~$s_{ij}$.
From here, the statement of the lemma is immediate.
\qed

\bibliographystyle{amsalpha}
\bibliography{bibliosplice}

\end{document}